\documentclass[english]{article}
\usepackage[T1]{fontenc}
\usepackage[latin9]{inputenc}
\usepackage{amsmath}
\usepackage{amsthm}
\usepackage{amssymb}
\usepackage{esint}
\usepackage{cite}
\usepackage{makeidx}

\date{\vspace{-3ex}}

\makeatletter
\theoremstyle{plain}
\newtheorem{theorem}{\protect\theoremname}
  \theoremstyle{plain}
  \newtheorem{conjecture}{\protect\conjecturename}
  \theoremstyle{definition}
  \newtheorem{problem}{\protect\problemname}
  \theoremstyle{remark}
  \newtheorem*{acknowledgement*}{\protect\acknowledgementname}
  \theoremstyle{plain}
  \newtheorem{lemma}{\protect\lemmaname}
  \theoremstyle{plain}
  \newtheorem{proposition}{\protect\propositionname}
  \theoremstyle{remark}
  
  \theoremstyle{plain}
  \newtheorem{corollary}{\protect\corollaryname}

\DeclareMathOperator{\ldim}{\underline{dim}}
\DeclareMathOperator{\edim}{dim_e}
\DeclareMathOperator{\bdim}{dim_\textrm{B}}
\DeclareMathOperator{\ubdim}{\overline{dim}_\textrm{B}}
\DeclareMathOperator{\uedim}{\overline{dim}_e}
\DeclareMathOperator{\ledim}{\underline{dim}_e}
\DeclareMathOperator{\id}{id}
\DeclareMathOperator{\supp}{supp}
\DeclareMathOperator*{\essinf}{essinf}
\DeclareMathOperator{\sgn}{sgn}
\newcommand{\bigdot}{\mbox{\LARGE{$.$}}}

\makeatletter
\def\blfootnote{\xdef\@thefnmark{}\@footnotetext}
\makeatother

\makeatother

\usepackage{babel}
  \providecommand{\acknowledgementname}{Acknowledgement}
  \providecommand{\conjecturename}{Conjecture}
  \providecommand{\corollaryname}{Corollary}
  \providecommand{\lemmaname}{Lemma}
  \providecommand{\problemname}{Problem}
  \providecommand{\propositionname}{Proposition}
  \providecommand{\remarkname}{Remark}
\providecommand{\theoremname}{Theorem}

\begin{document}

\title{Some problems on the boundary of fractal geometry and additive combinatorics}

\author{Michael Hochman}

\maketitle
\begin{abstract}
This paper is an exposition, with some new applications, of our results
from \cite{Hochman2014,Hochman2015} on the growth of entropy of convolutions.
We explain the main result on $\mathbb{R}$, and derive, via a linearization
argument, an analogous result for the action of the affine group on
$\mathbb{R}$. We also develop versions of the results for entropy
dimension and Hausdorff dimension. The method is applied to two problems
on the border of fractal geometry and additive combinatorics. First,
we consider attractors $X$ of compact families $\Phi$ of similarities
of $\mathbb{R}$. We conjecture that if $\Phi$ is uncountable and
$X$ is not a singleton (equivalently, $\Phi$ is not contained in
a 1-parameter semigroup) then $\dim X=1$. We show that this would
follow from the classical overlaps conjecture for self-similar sets,
and unconditionally we show that if $X$ is not a point and $\dim\Phi>0$
then $\dim X=1$. Second, we study a problem due to Shmerkin and Keleti,
who have asked how small a set $\emptyset\neq Y\subseteq\mathbb{R}$
can be if at every point it contains a scaled copy of the middle-third
Cantor set $K$. Such a set must have dimension at least $\dim K$
and we show that its dimension is at least $\dim K+\delta$ for some
constant $\delta>0$. 
\end{abstract}

\section{Introduction}

\subsection{Attractors of infinite compact families of similarities}

Let $G$ denote the group of similarities (equivalently, affine maps)
of the line and $\mathcal{S}\subseteq G$ the semigroup of contracting
similarities. Given a family $\Phi\subseteq\mathcal{S}$, a set $X\subseteq\mathbb{R}$
is called the attractor of $\Phi$ if it a compact non-empty set satisfying
\begin{equation}
X=\bigcup_{\varphi\in\Phi}\varphi(X).\label{eq:Hutchison}
\end{equation}
We then say that $\Phi$ generates $X$. 

Hutchinson's theorem \cite{Hutchinson1981} tells us that when $\Phi$
is finite, an attractor exists, and is unique. In this case $X$ is
said to be self-similar. These sets represent the simplest ``fractal''
sets, and their small-scale geometry has been extensively studied
over the years. It is also natural to ask what happens if we allow
$\Phi$ to be infinite. This opens the door to various ``pathologies'',
including the possibility that no attractor exists at all (i.e. the
only compact set satisfying (\ref{eq:Hutchison}) is the empty set),
and there is an extensive literature devoted to the case that $\Phi$
is countable, primarily concerning the case when $\Phi$ is unbounded
(see e.g. \cite{MauldinUrbanski1996}). But there is another generalization,
perhaps closer in spirit to the finite case, in which one takes $\Phi\subseteq\mathcal{S}$
to be compact. As in the finite case, existence and uniqueness of
the attractor is proved by showing that the map $Y\mapsto\bigcup_{\varphi\in\Phi}\varphi(Y)$
is a contraction on the space of compact subsets of $\mathbb{R}$,
endowed with the Hausdorff metric, and hence has a unique fixed point.\footnote{If we work in a complete metric space and general contractions, taking
the topology of uniform convergence of maps on compact sets, then
existence and uniqueness of the attractor is still true if we assume
that all $\varphi\in\Phi$ contract by at least some $0<r<1$. Without
uniformity, existence can fail already for a single map.}

We are interested in the dimension of the attractors of compact families
$\Phi$. For finite $\Phi$ this problem has a long history, which
we shall not recount here, but we note that even in this case our
understanding is still incomplete (see below and \cite{Hochman2014}).
However, there is some reason to believe that when $\Phi$ is compact
and uncountable the situation may be, in a sense, simpler.
\begin{conjecture}
\label{conj:main}If $X\subseteq\mathbb{R}$ is the attractor of an
uncountable compact family of contracting similarities, then either
$X$ is a single point, or $\dim X=1$.
\end{conjecture}
Some evidence for the conjecture is the fact that it is implied by
another well-known conjecture about the dimension of self-similar
sets. Recall that given a finite set $\Phi\subseteq\mathcal{S}$,
the similarity dimension $s=s(\Phi)$ of $\Phi$ (and, by convention,
of its attractor) is the solution to
\[
\sum_{\varphi\in\Phi}\left\Vert \varphi\right\Vert ^{s}=1,
\]
where $\left\Vert \varphi\right\Vert $ denotes the unsigned contraction,
or optimal Lipschitz constant, of $\varphi$. The similarity dimension
is an upper bound on the dimension of the attractor $X$, and we always
have $\dim X\leq1$ (because $X\subseteq\mathbb{R}$) so there is
the upper bound
\begin{equation}
\dim X\leq\min\{1,s(\Phi)\}.\label{eq:trivial-dim-bound}
\end{equation}
The inequality can be strict, but it is believed that this can happen
only for algebraic reasons. Specifically, let us say that $\Phi\subseteq G$
is free if the elements of $\Phi$ freely generate a free semigroup,
that is, if $\varphi_{1},\ldots,\varphi_{m},\psi_{1},\ldots,\psi_{n}\in\Phi$
then $\varphi_{1}\cdot\ldots\cdot\varphi_{m}=\psi_{1}\cdot\ldots\cdot\psi_{n}$
implies $m=n$ and $\varphi_{i}=\psi_{i}$ for all $i$.
\begin{conjecture}
[e.g. \cite{Simon1996}]\label{conj:overlaps}If $\Phi\subseteq\mathcal{S}$
is finite and free then its attractor $X$ has dimension $\dim X=\min\{1,s(\Phi)\}$.\end{conjecture}
\begin{theorem}
\label{thm:implication-btwn-conjectures}Conjecture \ref{conj:overlaps}
implies Conjecture \ref{conj:main}.
\end{theorem}
The proof of this implication relies on algebraic considerations,
namely, that for any ``large'' enough $\Phi\subseteq G$ there are
infinite free subsets of $\Phi$ (or of $\Phi^{k}$ for some $k$).
This result is in the same spirit as the classical Tits alternative,
which asserts that if a subgroup of a linear group is not virtually
solvable, then it contains free subgroups. Of course, we are working
in the affine group of the line, which is itself solvable, and so
cannot contain free subgroups at all; but this does not preclude the
existence of free semigroups, and this is what we need. Once we have
a large free semigroup, Conjecture \ref{conj:overlaps} ensures that
the attractor has large dimension. For a precise statements and proof
see Section \ref{sec:Conjecture-implies-conjecture}.

We are not able to prove conjecture \ref{conj:main}, but we give
the following result in its direction, where uncountability is replaced
by positive dimension:
\begin{theorem}
\label{thm:main}Let $X\subseteq\mathbb{R}$ be the attractor of a
compact family $\Phi\subseteq S$. If $\dim\Phi>0$, then either $X$
is a point, or $\dim X=1$.
\end{theorem}
We say a little about the proof later in this introduction, but before
moving on, let us mention a related and intriguing variant of the
conjecture, where uncountability is downgraded to ordinary infinity:
\begin{problem}
If $\Phi\subseteq\mathcal{S}$ is an infinite compact family of similarities,
and its attractor $X$ is not a single point, is $\dim X=1$?
\end{problem}
Of course if this were true, it would imply Conjecture \ref{conj:main}.
But this problem has the advantage that one can restrict it to algebraic
$\Phi$, that is, families $\Phi\subseteq\mathcal{S}$ all of whose
coordinates in the standard parametrization lie in a common algebraic
field. Note that Conjecture \ref{conj:overlaps} was shown in \cite{Hochman2014}
to hold under a similar assumption, and the methods of that paper
reduce the problem above to one about the random-walk entropy of families
of similarities of bounded contraction in a given algebraic group.
It seems possible that either a proof or counter-example can be found
in this setup.

There are other strengthenings of Conjecture \ref{conj:main}: for
example, instead of $\dim X=1$ one may ask if $X$ has positive Lebesgue
measure, or even non-empty interior. These problems are quite natural,
but seem out of reach of current methods.

\subsection{Sets containing many copies of the Cantor set}

Our second subject concerns the following problem. Let $K$ denote
the middle-$1/3$ Cantor set, translated so that it is symmetric around
the origin.
\begin{problem}
\label{Prob:Keleti-1}Given $s>0$, how large must a set $Y$ be if
there is a compact $C\subseteq\mathbb{R}$ with $\dim C=s$ and $Y$
contains a scaled copy of $K$ centered at each $c\in C$? In particular,
for $C=\mathbb{R}$, must we have $\dim Y=\min\{1,\dim C+\dim K\}$?
\end{problem}
I first learned of this problem from P. Shmerkin and T. Keleti. It
is related to problems on maximal operators on fractal sets, studied
by \L{}aba and Pramanik \cite{LabaPramanik2011}. It also is a relative
of the Furstenberg ``$\alpha$-set'' problem. 
\begin{theorem}
\label{thm:centered-Cantor-sets}Let $C\subseteq\mathbb{R}$ be compact
and of positive dimension. If $Y\subseteq\mathbb{R}$ contains a scaled
copy of $K$ centered at $c$ for every $c\in C$, then $\dim Y>\dim K+\delta$
where $\delta>0$ depends only on $\dim C$ and $\dim K$.
\end{theorem}
There is nothing special about the middle-third Cantor set; our argument
works when $K$ is any porous set. We note that A. M\'ath\'e recently
observed that, using a projection theorem due to Bourgain, one can
deduce that $\dim Y\geq\dim C/2$ \cite[Corollary 3.5]{Keleti2016}.
This gives better bounds than the theorem above in some cases, though
never when $\dim K>1/2$. Also worth noting is that for general sets
$K$, the last part of the problem (for the case $C=\mathbb{R}$)
has a negative answer, as shown by recent examples by A. M\'ath\'e
\cite[Theorem 3.2]{Keleti2016}. But for self-similar sets such as the
middle-$1/3$ Cantor set the question remains open and little seems
to be known. For a discussion of the history and related results see
\cite{Keleti2016}.

\subsection{\label{sub:The-role-of-additive-combinatorics}The role of additive
combinatorics}

Both of the problems above involve analysis of ``product'' sets,
where the product operation is the action of $G$ on $\mathbb{R}$.
Specifically, let $\varphi\bigdot x$ denote the image of $x\in\mathbb{R}$ under $\varphi\in G$, and for $X\subseteq\mathbb{R}$
and $\Phi\subseteq G$ denote 
\[
\Phi\bigdot X=\{\varphi\bigdot x\,:\,x\in X\,,\,\varphi\in\Phi\}
\]
A large part of this paper is devoted to studying how the ``size''
of $\Phi\bigdot X$ is related to the ``sizes'' and structure of $\Phi$ and $X$. This subject belongs to the field of additive
combinatorics, but we will not go into its history here. Rather, in
the coming paragraphs we outline, in an informal way, the main ideas
that we will encounter in the formal development later on. We emphasize
that the discussion below is heuristic and contains several half-truths,
which will be corrected later.

The leading principle is that $\Phi\bigdot X$ should be substantially
larger than $X$, unless there is some compatibility between the structure
of $X$ and $\Phi$. To explain the phenomenon we begin with the analogous
problem for sums of sets in the line, and choose Hausdorff dimension
as our measure of size. Thus, suppose that $\emptyset\neq X,Y\subseteq\mathbb{R}$
and consider their sum 
\[
X+Y=\{x+y\,:\,x\in X\,,\,y\in Y\}.
\]
It is clear that $\dim(Y+X)\geq\dim X$, since $Y+X$ contains a translate
of $X$. Equality of the dimensions can occur in two trivial ways:
(a) if the dimension of $X$ is maximal (that is, $\dim X=1$), or
(b) if the dimension of $Y$ is minimal (that is, $\dim Y=0$). Besides
the trivial cases there are many other non-trivial examples in which
$\dim(Y+X)=\dim X$ occurs, see e.g. \cite{ErdosVolkmann1966}.
However, when this happens, it turns out that the lack of dimension
growth can be explained by the approximate occurrence of (a) and (b)
for ``typical'' small ``pieces'' of the sets. To make this a little
more precise, define a scale-$r$ piece of $X$ to be a set of the
form $X\cap B_{r}(x)$ for some $x\in X$. The statement is then that,
if $\dim(Y+X)=\dim X$, then, roughly speaking, for typical scales
$0<r<1$, either (a) holds approximately for typical scale-$r$ pieces
of $X$, or (b) holds approximately for typical scale-$r$ pieces
of $Y$. The precise version of this, which is stated for measures
rather than sets, was proved in \cite{Hochman2014}; we state a variant
of it in Theorem \ref{thm:entropy-growth-under-convolution} below,
and use it as a black box. We remark that closely related results
appear in the work of Bourgain, e.g. \cite{Bourgain2003}.

Returning now to the action of $G$ on $\mathbb{R}$, suppose that
$\emptyset\neq X\subseteq\mathbb{R}$ and $\emptyset\neq\Phi\subseteq G$.
Then we again always have $\dim(\Phi\bigdot X)\geq\dim X$, and equality
can be explained by the same global reasons (a) and (b) above. But
there is also a third possibility, namely, (c) that $X$ is a point
and $\Phi$ is contained in the group of similarities fixing that
point. As with sumsets, $\dim(\Phi\bigdot X)=\dim X$ can also occur
in other ways, but it again turns out that if this happens then
the trivial explanations still apply to typical ``pieces'' of the
sets; thus at typical scales $0<r<1$, either (a) holds approximately
for typical scale-$r$ pieces of $X$, or (b) holds approximately
for typical scale-$r$ pieces of $\Phi$, or (c) holds approximately
for typical pairs of scale-$r$ pieces of $X$ and $\Phi$. 

But possibility (c) does not in reality occur unless $X$ is extremely 
degenerate. For suppose in the situation above that (c) holds at
some scale $r$. Then for typical pairs $\varphi\in\Phi$ and $x\in X$
we would have that $\Phi\cap B_{r}(\varphi)$ is approximately contained
in the stabilizer of $x$. But the $G$-stabilizers of different $x\in X$
are transverse (as submanifolds of $G$) so, assuming $X$ is infinite
(or otherwise large so as to ensure that no single point in it is
``typical''), by ranging over the possible values of $x$, we would
find that $\Phi\cap B_{r}(\varphi)$ is approximately contained in
the intersection of many transverse manifolds, hence is approximately
a point; and we are in case (b) again. In summary, if $X$ is large
enough, case (c) can be deleted from the list, leaving only (a) and
(b).

We shall prove the statements in the last two paragraph (in their
correct, measure formulation) in Section \ref{sec:Linearization-and-entropy-growth-for-action}.
But we note here that they are derived from the aforementioned result about sumsets, using a linearization
argument. To give some idea of how this works, let $f:G\times\mathbb{R}\rightarrow\mathbb{R}$
denote the action map $f(\varphi,x)=\varphi\bigdot x$, so that $\Phi\bigdot X=f(\Phi\times X)$.
Consider small pieces $X'=X\cap B_{r}(x)$ and $\Phi'=\Phi\cap B_{r}(\varphi)$
of $X,\Phi$, respectively. Then $\Phi'\bigdot X'=f(\Phi'\times X')$
is a subset of $\Phi\bigdot X$, and one can show that if we assume
that $\dim\Phi\bigdot X=\dim X$, then typical choices of $X'\bigdot\Phi'$
also satisfy $\dim\Phi'\bigdot X'=\dim X'$, approximately. But, since
$f$ is differentiable, for small $r$, the map $f$ is very close
to linear on the small ball $B_{r}(\varphi)\times B_{r}(x)$, and
hence $f(\Phi'\times X')$ is very close to the sumset $(\frac{\partial d}{d\varphi}f_{\varphi,x})\Phi'+(\frac{\partial}{\partial x}f_{\varphi,x})X'$
(here the subscript is the point at which the derivative is evaluated, and $\frac{\partial d}{d\varphi}f_{\varphi,x}$ is a $1\times 2$-matrix, so the sum is a sum of sets in the line). To this sum we can apply our results on sumsets. We remark that
the first term $\frac{\partial d}{d\varphi}f_{\varphi,x}(\Phi')$
in the sum may be substantially smaller than $\Phi'$, but if this
is the case it is because $\Phi'$ is essentially contained in the
kernel of $\frac{\partial d}{d\varphi}f_{\varphi,x}$, and this corresponds
to the case (c) above.

All this can be used to prove that, under mild assumptions, $\Phi\bigdot X$
is substantially larger than $X$. To be concrete, for $0<c<1$ let us say that
$X$ is $c$-porous if every interval $I\subseteq\mathbb{R}$
contains a sub-interval $J\subseteq I\setminus X$ of length $|J|=c|I|$.
We then have
\begin{theorem}
\label{thm:convolution-porous-sets}For any $0<c<1$ there exists
a $\delta=\delta(c)>0$ such that for any $c$-porous set
$X\subseteq\mathbb{R}$ of positive dimension, and any $\Phi\subseteq G$
with $\dim\Phi>c$, we have
\[
\dim\Phi\bigdot X\geq\dim X+\delta.
\]

\end{theorem}
Unlike the previous discussion this theorem is true as stated, see
Section \ref{sub:Proof-of-porousity-theorem}. Nevertheless let us
explain how it follows from our heuristic discussion. Suppose that
$\dim Y\bigdot X=\dim X$; then for typical scales $r$, either (a)
applies to typical scale-$r$ pieces of $X$, or (b) applies to typical
scale-$r$ pieces of $\Phi$. Suppose that $X$ is porous; then no
scale-$r$ piece $X\cap B_{r}(x)$ of $X$ can be close to a set of full dimension,
since porosity means that it contains a hole proportional in size
to $r$, and at every smaller scale. This rules out (a), so the remaining possibility is that
$Y\cap B_{r}(y)$ is approximately zero dimensional for typical $y\in Y$.
But one can show that if this is true for typical pieces of $Y$ at
typical scales, then it is true globally, i.e. $\dim Y=0$, as desired. 

To conclude this section let us explain how Theorem \ref{thm:convolution-porous-sets}
is related to Theorems \ref{thm:main} and \ref{thm:centered-Cantor-sets}.
In the first of these, the assumption is that $X=\Phi\bigdot X$ and
$\dim\Phi>0$. If $\dim X<1$ implied that  $X$ were porous, then
the theorem above would imply $\dim X>\dim X+\delta$, which is the
desired a contradiction. In our setting $X$ need not
be porous (and a-posteriori cannot be), but by working with suitable
measures on $X$ we will be able to apply an analog of Theorem \ref{thm:convolution-porous-sets},
which gives the result. 

To see the connection with Theorem \ref{thm:centered-Cantor-sets}, 
suppose $Y\subseteq\mathbb{R}$ contains a scaled copy of $K$ centered
at every point in a set $C\subseteq\mathbb{R}$. Assume that $\dim C>0$. 
Let $\Phi$ denote the set of similarities $\varphi_{r,c}(x)=rx+c$
for which $c\in C$ and $\varphi(K)\subseteq Y$, so that $\Phi\bigdot K\subseteq Y$. Since $Y$ is closed,
also $\Phi$ is closed, and by assumption for every $c\in C$ there
exists at least one $0<r<1$ such that $\varphi_{r,c}\in\Phi$. The
map $\varphi_{r,c}\mapsto c$ is a Lipschitz map taking $\Phi$ onto 
$C$, so $\dim\Phi\geq\dim C>0$. 
Finally,  $K$ is porous, so by Theorem \ref{thm:convolution-porous-sets}
 $\dim Y\geq \dim(\Phi\bigdot K) > \dim K+\delta$ for some $\delta>0$, as claimed.

\subsection{Organization of the paper}

In Section \ref{sec:Dyadic-partitions,-components-and-entropy} we
set up some notation, defining dyadic partitions on $\mathbb{R}$
and $G$, and discussing Shannon entropy and its properties. In Section
\ref{sec:Entropy-growth-for-euclidean-convolutions} we define component
measures and their distribution, and formulate the theorem on dimension
growth of convolutions of measures on $\mathbb{R}$. In Section \ref{sec:Linearization-and-entropy-growth-for-action}
we give the linearization argument which leads to the analogous growth
theorem for convolutions $\nu\bigdot\mu$ of $\nu\in\mathcal{P}(G)$
and $\mu\in\mathcal{P}(\mathbb{R})$. In Section \ref{sec:Proof-of-main-Theorem}
we prove Theorem \ref{thm:main-for-measures}. In Section \ref{sec:Growth-of-Hausdorff-dim}
we develop results for the Hausdorff dimension of convolutions, proving
Theorem \ref{thm:convolution-porous-sets} (and in so doing, completing the proof of Theorem \ref{thm:centered-Cantor-sets}). In Section \ref{sec:Conjecture-implies-conjecture}
we prove the implication between Conjecture \ref{conj:overlaps} and
Conjecture \ref{conj:main}. Finally, in Section \ref{sec:Extensions-and-open-problems}
we discuss another variant of Conjecture \ref{conj:main} in the non-linear
setting.

\subsubsection*{Acknowledgement}
I am grateful to Boris Solomyak for useful discussions, and to Ariel
Rapaport and the anonymous referee for a careful reading and for many
comments on a preliminary version of the paper. Part of the work on this paper was conducted during the 2016 program ``Dimension and Dynamics'' at ICERM. This research was supported by ERC grant 306494.

\section{\label{sec:Dyadic-partitions,-components-and-entropy}Measures, dyadic
partitions, components and entropy}

We begin with some background on entropy which will be used in our
analysis of convolutions.

\subsection{Probability measures}

For a measurable space $X$ we write $\mathcal{P}(X)$ for the space of
probability measures on $X$. We always take the Borel structure when the underlying space is metric. Given a measurable map $f:X\rightarrow Y$
between measurable spaces and $\mu\in\mathcal{P}(X)$ let $f\mu\in\mathcal{P}(Y)$
denote the push-forward measure, $\nu=\mu\circ f^{-1}$. For a probability
measure $\mu$ and set $A$ with $\mu(A)>0$ we write
\begin{equation}
\mu_{A}=\frac{1}{\mu(A)}\mu|_{A}\label{eq:conditional-measure}
\end{equation}
for the conditional measure on $A$.

\subsection{\label{sub:Dyadic-partitions-and-entropy}Dyadic partitions }

The level-$n$ dyadic partitions $\mathcal{D}_{n}$ of $\mathbb{R}$
is given by 
\[
\mathcal{D}_{n}=\{[\frac{k}{2^{n}},\frac{k+1}{2^{n}})\,:\,k\in\mathbb{Z}\},
\]
and the level-$n$ dyadic partition of $\mathbb{R}^{d}$ by 
\[
\mathcal{D}_{n}^{d}=\{I_{1}\times\ldots\times I_{d}\,:\,I_{i}\in\mathcal{D}_{n}\}.
\]
The superscript is often suppressed. 

We parametrize $G$ as $\mathbb{R}^{2}$, identifying $(s,t)\in\mathbb{R}^{2}$
with $x\mapsto e^{s}x+t$, and define a metric on $G$ by pulling
back the Euclidean metric on $\mathbb{R}^{2}$. The importance of
this choice of parametrization\footnote{We use another parametrizations in Section \ref{sec:Conjecture-implies-conjecture},
but only there.} is that if $\varphi,\psi\in G$ and $d(\varphi,\psi)<C$ then the
translation parts of $\varphi,\psi$ differ by an additive constant
$C'$, and their contractions by a multiplicative constant $C''$,
with $C',C''$ depending only on $C$. Note that, locally, this metric is equivalent (in fact diffeomorphic) to any Riemmanian metric on $G$ so the notion of dimension in $G$ is not affected by this choice of parametrization.

We equip $G$ with the dyadic partition $\mathcal{D}_{n}^{G}=\mathcal{D}_{n}^{2}$
induced from $\mathbb{R}^{2}$. 

When $t$ is not an integer, we write $\mathcal{D}_{t}=\mathcal{D}_{[t]}$
and $\mathcal{D}_{t}^{G}=\mathcal{D}_{[t]}^{G}$.

\subsection{\label{sub:Entropy}Entropy}

The Shannon entropy of a probability measure $\mu$ with respect to
a finite or countable partition $\mathcal{E}$ is defined by
\[
H(\mu,\mathcal{E})=-\sum_{E\in\mathcal{E}}\mu(E)\log\mu(E),
\]
The logarithm is in base $2$ and by convention $0\log0=0$. This
quantity is non-negative and we always have
\begin{equation}
H(\mu,\mathcal{E})\leq\log\#\{E\in\mathcal{E}\,:\,\mu(E)>0\}.\label{eq:entropy-counting-bound}
\end{equation}
The conditional entropy with respect to another countable partition
$\mathcal{F}$ is
\begin{equation}
H(\mu,\mathcal{E}|\mathcal{F})=\sum_{F\in\mathcal{F}}\mu(F)\cdot H(\mu_{F},\mathcal{E}),\label{eq:conditional-entropy-as-average}
\end{equation}
where $\mu_{F}$ is the conditional measure on $F$ (see (\ref{eq:conditional-measure})),
which is undefined when $\mu(F)=0$ but in that case its weight in the
sum is zero and it is ignored. Writing $\mathcal{E}\lor\mathcal{F}=\{E\cap F\,:\,E\in\mathcal{E}\,,\,F\in\mathcal{F}\}$
for the smallest common refinement of $\mathcal{E},\mathcal{F}$,
it is a basic identity that
\[
H(\mu,\mathcal{E}\lor\mathcal{F})=H(\mu,\mathcal{E}|\mathcal{F})+H(\mu,\mathcal{F}).
\]
Note that when $\mathcal{E}$ refines $\mathcal{F}$ (i.e. when every
atom of $\mathcal{E}$ is a subset of an atom of $\mathcal{F}$) we
have
\[
H(\mu,\mathcal{E}|\mathcal{F})=H(\mu,\mathcal{E})-H(\mu,F).
\]
In general, we always have
\[
H(\mu,\mathcal{E}|\mathcal{F})\leq H(\mu,\mathcal{E}),
\]
hence
\[
H(\mu,\mathcal{E}\lor\mathcal{F})\leq H(\mu,\mathcal{E})+H(\mu,\mathcal{F}).
\]

Entropy is concave, and almost convex: If with $\mu_{1},\mu_{2}$
probability measures and $\nu=\alpha\mu_{1}+(1-\alpha)\mu_{2}$
for some $0\leq\alpha\leq1$, then 
\[
\alpha H(\mu_{1},\mathcal{E})+(1-\alpha)H(\mu_{2},\mathcal{E})\leq H(\nu,\mathcal{E})\leq\alpha H(\mu_{1},\mathcal{E})+(1-\alpha)H(\mu_{2},\mathcal{E})+H(\alpha),
\]
where $H(\alpha)=-\alpha\log\alpha-(1-\alpha)\log(1-\alpha)$. The
same holds when all entropies above are conditional on a partition
$\mathcal{F}$. More generally if $\mu=\mu^{\omega}$ is a random
measure ($\omega$ denoting the point in the sample space), then\footnote{We require that $\omega\rightarrow\mu^{\omega}\in\mathcal{P}(X)$
be measurable in the sense that $\omega\mapsto\int fd\mu^{\omega}$
is measurable for all bounded measurable $f:X\rightarrow\mathbb{R}$,
and the expectation $\mathbb{E}(\mu)$ is understood the probability
measure $\nu$ determined by $\nu(A)=\mathbb{E}(\mu(A))$ for
all measurable $A$, or equivalently, $\int fd\nu=\mathbb{E}(\int fd\mu)$
for bounded measurable $f$.}
\[
H(\mathbb{E}(\mu),\mathcal{E})\geq\mathbb{E}\left(H(\mu,\mathcal{E})\right),
\]
and similarly for conditional entropies.

\subsection{\label{sub:Expansion-and-re-scaling}\label{sub:Behavior-of-entropy-under-similarities}Translation,
scaling and their effect on entropy}

Define the translation map $T_{u}:\mathbb{R}\rightarrow\mathbb{R}$
by 
\[
T_{u}(x)=x+u,
\]
and the scaling map $S_{t}:\mathbb{R}\rightarrow\mathbb{R}$ by 
\[
S_{t}x=2^{t}x.
\]
Note our choice of parametrization, and that $S_{s+t}=S_{s}S_{t}$.

It is clear that if $k\in\mathbb{Z}$ then 
\[
H(S_{k}\mu,\mathcal{D}_{n-k})=H(\mu,\mathcal{D}_{n}),
\]
because $S_{k}$ maps the atoms of $\mathcal{D}_{n}$ to the atoms
of $\mathcal{D}_{n-k}$. When $t$ is not a power of $2$ the same
relation holds, but with an error term:
\begin{equation}
H(S_{t}\mu,\mathcal{D}_{n-t})=H(\mu,\mathcal{D}_{n})+O(1).\label{eq:entropy-under-scaling}
\end{equation}
Translation affects entropy in a similar way: if $u=m/2^{n}$ for
$m,n\in\mathbb{Z}$ then 
\[
H(T_{u}\mu,\mathcal{D}_{n})=H(\mu,\mathcal{D}_{n}),
\]
and for general $u\in\mathbb{R}$,
\begin{equation}
H(T_{u}\mu,\mathcal{D}_{n})=H(\mu,\mathcal{D}_{n})+O(1).\label{eq:entropy-under-translation}
\end{equation}
Combining all this we find that if $\varphi$ is a similarity and
$\left\Vert \varphi\right\Vert $ is its unsigned contraction constant (it optimal Lipschitz constant), then 
\begin{equation}
H(\varphi\mu,\mathcal{D}_{n-\log\left\Vert \varphi\right\Vert })=H(\mu,\mathcal{D}_{n})+O(1).\label{eq:entorpy-under-similarity}
\end{equation}

If $\mu$ is supported on a set of diameter $O(1)$, then by (\ref{eq:entropy-counting-bound}),
$H(\mu,\mathcal{D}_{1})=O(1)$. It follows from the above that if
$\mu$ is supported on a set of diameter $2^{-(n+c)}$. Then
\begin{equation}
H(\mu,\mathcal{D}_{n})=O_{c}(1),\label{eq:entorpy-combinatorial-bound}
\end{equation}
and in particular, for $m>n$,
\begin{equation}
H(\mu,\mathcal{D}_{m}|\mathcal{D}_{n})=H(\mu,\mathcal{D}_{m})-O_{c}(1).\label{eq:conditional-entropy-combinatorial-bound}
\end{equation}

Finally, although entropy is not quite continuous under small perturbations
of the measure, it almost is. Specifically, let $\mu\in\mathcal{P}(\mathbb{R})$. If a function $f$ satisfies $c^{-1}d(x,y)\leq d(f(x),f(y))\leq cd(x,y)\leq c$, then
then
\begin{equation}
H(\varphi\mu,\mathcal{D}_{n})=H(\mu,\mathcal{D}_{n})+O(\log c).\label{eq:entropy-under-bi-lip-maps}
\end{equation}
and if $f,g:\mathbb{R}\rightarrow\mathbb{R}$ are $2^{-n}$-close
to in the sup-distance (i.e. $|f(x)-g(x)|<2^{-n}$ for all $x$),
then
\begin{equation}
|H(f\mu,\mathcal{D}_{n})-H(g\mu,\mathcal{D}_{n})|=O(1).\label{eq:entroupy-distortion-bound}
\end{equation}

\section{\label{sec:Entropy-growth-for-euclidean-convolutions}Entropy growth
for Euclidean convolutions}

The convolution $\nu*\mu$ of $\nu,\mu\in\mathcal{P}(\mathbb{R})$
is the push-forward of $\nu\times\mu$ by the map $(x,y)\mapsto x+y$.
In this section we state a result from \cite{Hochman2014} saying
that convolution increases entropy, except when some special structure
is present.  The statement is in terms of the multi-scale structure
of the measures, and we first develop the language necessary for describing
it.

\subsection{\label{sub:Component-measures}Component measures }

For $x\in\mathbb{R^{d}}$ let $\mathcal{D}_{n}(x)=\mathcal{D}_{n}^{d}(x)$
denote the unique element of $\mathcal{D}_{n}^{d}$ containing it,
and for a measure $\mu\in\mathcal{P}(\mathbb{R}^{d})$ define the
level-$n$ component of $\mu$ at $x$ to be the conditional measure
on $\mathcal{D}_{n}(x)$: 
\[
\mu_{x,n}=\mu_{\mathcal{D}_{n}(x)}=\frac{1}{\mu(\mathcal{D}_{n}(x))}\mu|_{\mathcal{D}_{n}(x)}.
\]
This is defined for $\mu$-a.e.\ $x$. 

We define components of a measure $\nu\in\mathcal{P}(G)$ in the
same way, using the dyadic partitions $\mathcal{D}_{n}^{G}$, so $\nu_{g,n}=\frac{1}{\nu(\mathcal{D}_{n}^{G}(g))}\nu|_{\mathcal{D}_{n}^{G}(g)}$.

\subsection{\label{sub:Random-component-measures}Random component measures }

We often view $\mu_{x,n}$ as a random variable, with $n$ chosen
uniformly within some specified range, and $x$ chosen according to
$\mu$, independently of $n$. This is the intention whenever $\mu_{x,n}$
appears in an expression $\mathbb{P}(\ldots)$ or $\mathbb{E}(\ldots)$. 

An equivalent way of generating $\mu_{x,i}$ is to choose $i\in\{0,\ldots,n\}$
uniformly, and independently choose $I\in\mathcal{D}_{i}$ with probability
$\mu(I)$. Then the random measure $\mu_{I}$ has the same distribution
as $\mu_{x,i}$, and if we further then choose $x\in I$ using the
measure $\mu_{I}$, then the distribution of $\mu_{x,i}$ generated
in this way agrees with the previous procedure.

For example, if $\mathcal{U}$ is a set of measures then $\mathbb{P}_{0\leq i\leq n}(\mu_{x,i}\in\mathcal{U})$
is the probability that $\mu_{x,i}\in\mathcal{U}$ when $i\in\{0,\ldots,n\}$
is chosen uniformly, and $x$ is independently chosen according to
$\mu$.

Similarly, $\mathbb{E}_{i=n}(H(\mu_{x,i},\mathcal{D}_{i+m}))$ denotes
the expected entropy of a component at level $n$ (note that we took
$i=n$, so the level is deterministic), measured at scale $n+m$,
and by (\ref{eq:conditional-entropy-as-average}), 
\[
H(\mu,\mathcal{D}_{n+m}|\mathcal{D}_{n})=\mathbb{E}_{i=n}\left(H(\mu_{x,i},\mathcal{D}_{n+m})\right).
\]
As another example, we have the trivial identity
\[
\mu=\mathbb{E}_{i=n}(\mu_{x,i}).
\]

We view components of measures on $G$ as random variables in the
same way as above and adopt the same notational conventions. 

Our notation defines $x$ and $i$ implicitly as random variables.
For example we could write
\[
\mathbb{P}_{0\leq i\leq n}(H(\mu_{x,i},\mathcal{D}_{i+1})=1\mbox{ and }i\geq n_{0})
\]
for the probability that a random component has full entropy at one
scale finer, and the scale is at least $n_{0}$.

When several random components are involved, they are assumed to be
chosen independently unless otherwise specified. Thus the distribution
of $\nu_{g,i}\times\mu_{x,i}$ is obtained by choosing $i$ first
and then choosing $g$ and $x$ independently according to $\nu$
and $\mu$, respectively. Note the resulting random measure has the
same distribution as $(\nu\times\mu)_{(g,x),i}$.

The distribution on components has the convenient property that it
is almost invariant under repeated sampling, i.e. choosing components
of components. More precisely, for a probability measure $\mu\in\mathcal{P}(\mathbb{R})$
and $m,n\in\mathbb{N}$, let $\mathbb{P}_{n}^{\mu}$ denote the distribution
of components $\mu_{x,i}$, $0\leq i\leq n$, as defined above; and
let $\mathbb{Q}_{n,m}^{\mu}$ denote the distribution on components
obtained by first choosing a random component $\mu_{x,i}$, $0\leq i\leq n$,
as above, and then, conditionally on $\nu=\mu_{x,i}$, choosing
a random component $\nu_{y,j}$, $i\leq j\leq i+m$ in the usual
way (note that $\nu_{y,j}=\mu_{y,j}$ is indeed a component of
$\mu$). 
\begin{lemma}
\label{lem:distribution-of-components-of-components}Given $\mu\in\mathcal{P}(\mathbb{R})$
and $m,n\in\mathbb{N}$, the total variation distance between $\mathbb{P}_{n}^{\mu}$
and $\mathbb{Q}_{n,m}^{\mu}$ satisfies 
\[
\left\Vert \mathbb{P}_{n}^{\mu}-\mathbb{Q}_{n,m}^{\mu}\right\Vert =O(\frac{m}{n}).
\]
In particular let $\mathcal{A}_{i},\mathcal{B}_{i}\subseteq\mathcal{P}([0,1)^{d})$,
write $\alpha=\mathbb{P}_{0\leq i\leq n}(\mu_{x,y}\in\mathcal{A}_{i})$,
and suppose that $\nu\in\mathcal{A}_{i}$ implies $\mathbb{P}_{i\leq j\leq i+m}(\nu_{x,j}\in\mathcal{B}_{j})\geq\beta$.
Then 
\[
\mathbb{P}_{0\leq i\leq n}(\mu_{x,i}\in\mathcal{B}_{i})>\alpha\beta-O(\frac{m}{n}).
\]

\end{lemma}
These are essentially applications of the law of total probability,
for details see \cite{HochmanSolomyak2016}.

\subsection{\label{sub:Multiscale-formulas-for-entropy}Multiscale formulas for
entropy }

Let us call $\frac{1}{n}H(\mu,\mathcal{D}_{n})$ the scale-$n$ entropy
of $\mu$. A simple but very useful property of scale-$n$ entropy
of a measure is that when $m\ll n$ it is roughly equal to the average
of the scale-$m$ entropies of its components, and for convolutions
a related bound can be given. The proofs can be found in \cite[Section 3.2]{Hochman2014}.

\begin{lemma}
\label{lem:multiscale-formula-for-entropy}For compactly supported
$\mu\in\mathcal{P}(\mathbb{R})$ or $\mu\in\mathcal{P}(G)$, for
every $m,n\in\mathbb{N}$, 
\begin{eqnarray*}
\frac{1}{n}H(\mu,\mathcal{D}_{n}) & = & \mathbb{E}_{1\leq i\leq n}\left(\frac{1}{m}H(\mu_{x,i},\mathcal{D}_{i+m})\right)+O(\frac{m}{n}).
\end{eqnarray*}
The error term depends only on the diameter of the support of $\mu$.
\end{lemma}
For convolutions in $\mathbb{R}$ we have a lower bound:
\begin{lemma}
\textup{\label{lem:multiscale-formula-for-entropy-of-additive-convolution}For
compactly supported $\mu,\nu\in\mathcal{P}(\mathbb{R})$, for
every $m,n\in\mathbb{N}$,
\begin{eqnarray*}
\frac{1}{n}H(\nu*\mu,\mathcal{D}_{n}) & \geq & \mathbb{E}_{1\leq i\leq n}\left(\frac{1}{m}H(\nu_{y,i}*\mu_{x,i},\mathcal{D}_{i+m})\right)-O(\frac{1}{m}+\frac{m}{n}).
\end{eqnarray*}
}The error term depends only the diameter of the supports of $\nu,\mu$.
\end{lemma}
In the expectations above, the random variables $\mu_{x,i}$ and
$\nu_{y,i}$ are independent components of level $i$.

Before we state the analogous formula for convolutions $\nu\bigdot\mu$
where $\nu\in\mathcal{P}(\mathcal{S})$ and $\mu\in\mathcal{P}(\mathbb{R})$,
we first explain how contraction enters the formula. When $\varphi\in\mathcal{S}$
acts on a measure $\mu\in\mathcal{P}(\mathbb{R})$, it contracts
$\mu$ by $\left\Vert \varphi\right\Vert $. By (\ref{eq:entorpy-under-similarity})
this implies that for any $i$, 
\begin{eqnarray*}
H(\mu,\mathcal{D}_{i}) & = & H(\varphi\mu,\mathcal{D}_{i-\log\left\Vert \varphi\right\Vert })+O(1)
\end{eqnarray*}
(note that $\log\left\Vert \varphi\right\Vert <0$ when $\varphi$
is a contraction). Thus if $\nu$ is a measure supported on a small
neighborhood of $\varphi$, then the entropy of $\nu\bigdot\mu$
should be measured at a resolution adjusted by $\log\left\Vert \varphi\right\Vert $-scales
relative to the resolution at which we consider $\mu$. The analog
of Lemma \ref{lem:multiscale-formula-for-entropy-of-additive-convolution}
now has the following form (see also \cite[Lemma 5.7]{Hochman2015}):
\begin{lemma}
\label{lem:multiscale-formula-for-entropy-of-action-convolution}For
compactly supported $\mu\in\mathcal{P}(\mathbb{R})$ and $\nu\in\mathcal{P}(\mathcal{S})$,
for every $\varphi_{0}\in\supp\nu$ and for every $m,n$,
\begin{eqnarray*}
  \frac{1}{n}H(\nu\bigdot\mu,\mathcal{D}_{n-\log\left\Vert \varphi_{0}\right\Vert })& \;\geq\; & \mathbb{E}_{1\leq i\leq n}(\frac{1}{m}H(\nu_{\varphi,i}\bigdot\mu_{x,i},\mathcal{D}_{i-\log\left\Vert \varphi_{0}\right\Vert +m}))\\
  & & \;-\;O(\frac{1}{m}+\frac{m}{n}).
\end{eqnarray*}
The error term depends only the diameter of the supports of $\mu,\nu$.
\end{lemma}
In our application of this inequality, the support of $\nu$ will
lie in a fixed compact set, and we can drop the scale-shift of $\log\left\Vert \varphi_{0}\right\Vert $
and absorb the change in the error term; that is we can replace $\mathcal{D}_{n-\log\left\Vert \varphi_{0}\right\Vert }$
by $\mathcal{D}_{n}$ and $\mathcal{D}_{i-\log\left\Vert \varphi_{0}\right\Vert +m}$
by $\mathcal{D}_{i+m}$.

\subsection{\label{sub:Essentially-bounded-component-component-entropy}Entropy
porosity }

For a general measure, the entropy of components may vary almost arbitrarily
from scale to scale and within a fixed scale. The following definition
imposes some degree of regularity, specifically, it prevents too many
components from being too uniform. 

Let $\mu\in\mathcal{P}(\mathbb{R})$. We say that $\mu$ is \emph{$(h,\delta,m)$-entropy
porous}\footnote{Entropy porosity in the sense above is
weaker than porosity, since it allows the measure to be fully supported on small
balls (i.e. there do not need to be holds in its support). But the upper bound on the entropy of components means that
most components are far away from being uniform at a slightly finer
scale.}\emph{ from scale $n_{1}$ to $n_{2}$ }if 
\begin{equation}
\mathbb{P}_{n_{1}\leq i\leq n_{2}}\left(\frac{1}{m}H(\mu_{x,i},\mathcal{D}_{i+m})\leq h+\delta\right)>1-\delta.\label{eq:bounded-component-entropy}
\end{equation}
We say that it is $h$-entropy porous if for every $\delta>0$, $m>m(\delta)$
and $n>n(\delta,m)$ the measure is $(h,\delta,m)$-entropy porous
from scale $0$ to $n$. 

Note that if $\mu$ is $(h,\delta,m)$-entropy porous from scale
$0$ to $n$ then by Lemma \ref{lem:multiscale-formula-for-entropy} we
have $H(\mu,\mathcal{D}_{n})/n\leq h+2\delta+O(m/n)$. 

We will use the fact that entropy porosity passes to components. More
precisely,
\begin{lemma}
\label{lem:bounded-component-entropy-passes-to-components}Let $0<\delta<1$,
$m,k\in\mathbb{N}$ and $n>n(\delta,k)$. If $\mu\in\mathcal{P}(\mathbb{R})$
is $(h,\delta^{2}/2,m)$-entropy porous from scale $0$ to $n$, then
\begin{equation}
\mathbb{P}_{0\leq i\leq n}\left(\begin{array}{c}
\mu_{x,i}\mbox{ is }(h,\delta,m)\mbox{-entropy }\\
\mbox{porous from scale }i\mbox{ to }i+k
\end{array}\right)>1-\delta.\label{eq:bounded-component-entropy-of-components}
\end{equation}
\end{lemma}
\begin{proof}
By assumption, 
\begin{equation}
\mathbb{P}_{0\leq i\leq n}\left(\frac{1}{m}H(\mu_{x,i},\mathcal{D}_{i+m})\leq h+\frac{\delta^{2}}{2}\right)>1-\frac{\delta^{2}}{2}.\label{eq:10}
\end{equation}
Let $\mathcal{B}_{i}\subseteq\mathcal{P}(\mathbb{R})$ denote the
set of measures $\nu$ with $\frac{1}{m}H(\nu,\mathcal{D}_{i+m})>h+\delta$,
and $\mathcal{A}_{i}\subseteq\mathcal{P}(\mathbb{R})$ the set of
$\nu$ such that $\mathbb{P}_{i\leq j\leq i+k}(\nu_{x,j}\in\mathcal{B}_{j})>\delta$.
It suffices for us to show that $\mathbb{P}_{0\leq i\leq n}(\mu_{x,i}\in\mathcal{A}_{i})\leq2\delta/3$.
Indeed, if we had $\mathbb{P}_{0\leq i\leq n}(\mu_{x,i}\in\mathcal{A}_{i})>2\delta/3$,
then Lemma \ref{lem:distribution-of-components-of-components} would
imply $\mathbb{P}_{0\leq i\leq n}(\mu_{x,i}\in\mathcal{B}_{i})\geq2\delta^{2}/3-O(k/n)$,
which, assuming as we may that $n$ large relative to $k,\delta$,
contradicts (\ref{eq:10}).
\end{proof}

\subsection{\label{sub:Entropy-growth-under-conv-Euclidean}Entropy growth under
convolution: Euclidean case}

Recall that $\nu*\mu$ denotes the convolution of measures $\nu,\mu$
on $\mathbb{R}$. The entropy of a convolution is generally at least
as large as each of the convolved measures, although due to the discretization
involved there may be a small loss:
\begin{lemma}
\label{lem:entropy-monotonicity-underconvolution}For every $\mu,\nu\in\mathcal{P}(\mathbb{R})$,
\[
\frac{1}{n}H(\nu*\mu,\mathcal{D}_{n})\geq\frac{1}{n}H(\mu,\mathcal{D}_{n})-O(\frac{1}{n}).
\]
\end{lemma}
\begin{proof}
Let $X$ be a random variable with distribution $\nu$. Then 
\begin{eqnarray*}
\nu*\mu & = & \mathbb{E}(\delta_{X}*\mu)\\
 & = & \mathbb{E}(T_{X}\mu).
\end{eqnarray*}
By concavity of entropy and (\ref{eq:entropy-under-translation}),
\begin{eqnarray*}
H(\nu*\mu,\mathcal{D}_{n}) & \geq & \mathbb{E}\left(H(T_{X}\mu,\mathcal{D}_{n})\right)\\
 & = & \mathbb{E}\left(H(\mu,\mathcal{D}_{n})+O(1)\right)\\
 & = & H(\mu,\mathcal{D}_{n})+O(1).
\end{eqnarray*}
The lemma follows.
\end{proof}
In general one expects the entropy to grow under convolution but this
is not always the case. Theorem 2.8 of \cite{Hochman2014} provides
a verifiable condition under which some entropy growth occurs.
\begin{theorem}
\label{thm:entropy-growth-under-convolution}For every $\varepsilon>0$
there exists a $\delta=\delta(\varepsilon)>0$ such that for every
$m>m(\varepsilon,\delta)$ and $n>n(\varepsilon,\delta,m)$, the following
holds: 

Let $\mu,\nu\in\mathcal{P}([0,1))$ and suppose that $\mu$ is
$(1-\varepsilon,\delta,m)$-entropy porous from scale $0$ to $n$.
Then 
\[
\frac{1}{n}H(\nu,\mathcal{D}_{n})>\varepsilon\quad\implies\quad\frac{1}{n}H(\nu*\mu,\mathcal{D}_{n})>\frac{1}{n}H(\mu,\mathcal{D}_{n})+\delta.
\]
More generally, if $\mu,\nu$ are supported on sets of diameter
$2^{-i}$, and if $\mu$ is $(1-\varepsilon,\delta,m)$-entropy porous
from scale $i$ to $i+n$, then 
\[
\frac{1}{n}H(\nu,\mathcal{D}_{i+n})>\varepsilon\quad\implies\quad\frac{1}{n}H(\nu*\mu,\mathcal{D}_{i+n})>\frac{1}{n}H(\mu,\mathcal{D}_{i+n})+\delta.
\]

\end{theorem}
The second statement follows from the first by re-scaling by $2^{i}$.

\section{\label{sec:Linearization-and-entropy-growth-for-action}Linearization
and entropy growth }

We now consider $\nu\in\mathcal{P}(G)$ and $\mu\in\mathcal{P}(\mathbb{R})$
and the convolution $\nu\bigdot\mu$ obtained by pushing $\nu\times\mu$
forward through $(\varphi,x)\mapsto\varphi\bigdot x=\varphi(x)$. Our
goal is to extend the results of the last section to this case: namely,
that under some assumptions on $\nu,\mu$ the entropy of $\nu\bigdot\mu$
is substantially larger than that of $\mu$ alone.

It will be convenient to extend the notation and write $\nu\bigdot x$
for the push-forward of $\nu\in\mathcal{P}(G)$ via $\varphi\mapsto\varphi\bigdot x$,
or equivalently, $\nu\bigdot x=\nu\bigdot\delta_{x}$.

\subsection{\label{sub:Linearization}Linearization and entropy}

Let $f:\mathbb{R}^{d_{1}+d_{2}}\rightarrow\mathbb{R}^{d_{3}}$, let
$\nu\in\mathcal{P}(\mathbb{R}^{d_{1}})$, $\mu\in\mathcal{P}(\mathbb{R}^{d_{2}})$,
and $\lambda=f(\nu\times\mu)\in\mathcal{P}(\mathbb{R}^{d_{3}})$.
First suppose that $f$ is affine, so that there exists $y_{0}\in\mathbb{R}^{2}$
and matrices $A,B$ of appropriate dimensions such that 
\begin{eqnarray*}
f(x,y) & = & y_{0}+Ax+By\\
 & = & T_{y_{0}}(Ax+By).
\end{eqnarray*}
It follows that
\[
\lambda=f(\nu\times\mu)=T_{y_{0}}(A\nu*B\mu),
\]
and by (\ref{eq:entropy-under-translation}),
\[
H(\lambda,\mathcal{D}_{n})=H(A\nu*B\mu,\mathcal{D}_{n})+O(1).
\]

Now suppose instead that $f$ is twice continuously differentiable,\footnote{Differentiability would be enough for most purposes, but then the
error term in (\ref{eq:linearization}) would be merely $o(|x-x_{0}|+|y-y_{0}|)$
instead of the quadratic error, and later on we will want the quadratic rate.}
rather than affine, so at every point $z_{0}=(x_{0},y_{0})\in\mathbb{R}^{d_{1}+d_{2}}$
there are matrices $A=A_{z_{0}}$ and $B=B_{z_{0}}$ such that 
\begin{align}
  f(x,y)  \;= & \;\; f(x_{0},y_{0})+A(x-x_{0})+B(y-y_{0}) \nonumber \\
      & \;\;\;+\;\;O(|x-x_{0}|^{2}+|y-y_{0}|^{2}).\label{eq:linearization}
\end{align}
Fix $m$, and suppose further that $r>0$ and that $\nu$ is supported
on an $O(r)$-neighborhood $U$ of $x_{0}$ and $\mu$ is supported
on an $O(r)$-neighborhood $V$ of $y_{0}$. Then, assuming $r$ is
small enough that the error term in (\ref{eq:linearization}) is less
then $2^{\log r-m}$ for all $(x,y)\in U\times V$, by (\ref{eq:entroupy-distortion-bound})
we have 
\begin{eqnarray}
H(f(\nu\times\mu),\mathcal{D}_{-\log r+m}) & = & H(A\nu*B\mu,\mathcal{D}_{-\log r+m})+O(1).\label{eq:entropy-under-linearization}
\end{eqnarray}

The last equation shows that, in order to bound the entropy of the
image of a product measure, we can apply results about convolutions,
provided we control the error term. But the dependence between the
parameters is crucial: We have controlled it for a given $m$ by requiring
that $\nu\times\mu$ be supported close enough to $z_{0}$. In
(\ref{eq:entropy-under-linearization}) \emph{we cannot take $m\rightarrow\infty$},
because as we increase $m$, the supports of the measures may be required
to shrink to a point. 

This issue can be avoided by using multiscale formula for entropy,
though this gives only a lower bound rather than equality. We specialize
at this point to the linear action of the similarity group $G$ on
$\mathbb{R}$, though the same ideas work in greater generality.
Recall that we parametrize $G$ as $\mathbb{R}^{2}$, identifying
$(s,t)$ with $x\mapsto e^sx+t$.
In order to conform with the notation in previous sections, we denote
the coordinates of $G\times\mathbb{R}$ by $(\varphi,x)$. Let
\begin{eqnarray*}
f:G\times\mathbb{R} & \rightarrow & \mathbb{R}\\
(\varphi,x) & \mapsto & \varphi\bigdot x
\end{eqnarray*}
denote the action map, which we think of this as a smooth map defined
on $\mathbb{R}^{2}\times\mathbb{R}$. Note that
by definition, $f(\mu\times\nu)=\mu.\nu$. Also note that
the derivative $A=A_{(\varphi,x)}=\frac{\partial}{\partial\varphi}f(\varphi,x)$
is a $1\times2$ matrix and $B=B_{(\varphi,x)}=\frac{\partial}{\partial x}f(\varphi,x)$
is a $1\times1$ matrix, which we identify simply with a scalar. Given
$(u,v)\in\mathbb{R}^{2}\times\mathbb{R}$ and matrices $A,B$ of these
dimensions we have (recall that $S_{t}(x)=2^{t}x$) 
\[
Au+Bv=S_{\log B}(B^{-1}Au+v).
\]
Therefore, 
\[
A\nu*B\mu=S_{\log B}(B^{-1}A\nu*\mu).
\]

\begin{proposition}
\label{prop:iterated-entropy-formula-for-action}Let $I\times J\subseteq G\times\mathbb{R}$
be compact. Then for every $\nu\in\mathcal{P}(I)$ and $\mu\in\mathcal{P}(J)$,
as $m\rightarrow\infty$ and $n/m\rightarrow\infty$, we have 
\[
\frac{1}{n}H(\nu\bigdot\mu,\mathcal{D}_{n})\geq\mathbb{E}_{0\leq i\leq n}\left(\frac{1}{m}H(B_{(\varphi,x)}^{-1}A_{(\varphi,x)}\nu_{\varphi,i}*\mu_{x,i},\mathcal{D}_{i+m})\right)+o(1).
\]
\end{proposition}
\begin{proof}
Since $I\times J$ is compact and $f$ is smooth, the error term in
(\ref{eq:linearization}) holds uniformly in $(\varphi_{0},x_{0})\in I\times J$.
Given components $\nu_{\varphi,i}$ and $\mu_{x,i}$ of $\nu$
and $\mu$, respectively, each is supported on a set of diameter
$O(2^{-i})$, so by (\ref{eq:entropy-under-linearization}),
\[
\frac{1}{m}H(f(\nu_{\varphi,i}\times\mu_{x,i}),\mathcal{D}_{i+m})=\frac{1}{m}H(A_{(\varphi,x)}\nu_{\varphi,i}*B_{(\varphi,x)}\mu_{x,i},\mathcal{D}_{i+m})+o(1),
\]
as $m,i\rightarrow\infty$ (uniformly in $(\varphi,x)\in I\times J$). By compactness,
$B_{(\varphi,x)}$ is bounded for $(\varphi,x)\in I\times J$, and
by (\ref{eq:entropy-under-scaling}), changing a measure by a bounded
scaling affects entropy by $O(1)$, which, upon division by $m$,
is $o(1)$. Thus the last equation can be replaced by 
\[
\frac{1}{m}H(f(\nu_{\varphi,i}\times\mu_{x,i}),\mathcal{D}_{i+m})=\frac{1}{m}H(B_{(\varphi,x)}^{-1}A_{(\varphi,x)}\nu_{\varphi,i}*\mu_{x,i},\mathcal{D}_{i+m})+o(1),
\]
as $m,i\rightarrow\infty$. Finally, by Lemma \ref{lem:multiscale-formula-for-entropy-of-action-convolution}
and the remark following it, and the last equation,
\begin{eqnarray*}
\frac{1}{n}H(\nu\bigdot\mu,\mathcal{D}_{n}) & \geq & \mathbb{E}_{0\leq i\leq n}\left(\frac{1}{m}H(\nu_{\varphi,i}\bigdot\mu_{x,i},\mathcal{D}_{i+m})\right)+O(\frac{1}{m}+\frac{m}{n})\\
 & = & \mathbb{E}_{0\leq i\leq n}\left(\frac{1}{m}H(f(\nu_{\varphi,i}\times\mu_{x,i}),\mathcal{D}_{i+m})\right)+O(\frac{1}{m}+\frac{m}{n})\\
& = & \mathbb{E}_{0\leq i\leq n}\left(\frac{1}{m}H(B_{(\varphi,x)}^{-1}A_{(\varphi,x)}\nu_{\varphi,i}*\mu_{x,i},\mathcal{D}_{i+m})+o(1)\right)\\
& & \; +\;O(\frac{1}{m}+\frac{m}{n}),
\end{eqnarray*}
which gives the claim (we can move the error term outside the expectation
because it is uniform).
\end{proof}

\subsection{\label{sub:Entropy-growth}Entropy growth for the action}

We now prove an analogue of Theorem \ref{thm:entropy-growth-under-convolution}
for the action of $G$ on $\mathbb{R}$.
\begin{theorem}
\label{thm:inverse-theorem-for-action}For every $\varepsilon>0$
there exists a $\delta=\delta(\varepsilon)>0$ such that the following
holds: 

Let $\nu\in\mathcal{P}(G)$, $\mu\in\mathcal{P}(\mathbb{R})$
be compactly supported, and suppose that $\mu$ is non-atomic and
$(1-\varepsilon)$-entropy porous. Then for every $n>n(\varepsilon,\delta,\mu)$
\[
\frac{1}{n}H(\nu,\mathcal{D}_{n})>\varepsilon\quad\implies\quad\frac{1}{n}H(\nu\bigdot\mu,\mathcal{D}_{n})>\frac{1}{n}H_{n}(\mu,\mathcal{D}_{n})+\delta.
\]

\end{theorem}
We remark that $n$ is required to be large relative to $\mu$, but
in fact the only dependence involves the modulus of continuity of
$\mu$ (in the proof the dependence appears in Lemma \ref{lem:entropy-of-image}),
and on a choice of the parameter $m$ in the definition of entropy
porosity for $\mu$.

To begin the proof, fix $\varepsilon>0$. Apply Theorem \ref{thm:entropy-growth-under-convolution}
with parameter $\varepsilon'=\varepsilon/10$, obtaining a corresponding
$\delta'>0$. We will choose $\delta$ later to be small both compared to
$\delta'$ and $\varepsilon$.

Fix parameters $m,k,n\in\mathbb{N}$. All the $o(1)$ error terms
below are to be understood as becoming arbitrarily small if $m$ is
large, $k$ is large enough depending on $m$, and $n$ is large enough
in a manner depending on $m,k$. 

Let us abbreviate
\[
C_{(\varphi,x)}=B_{(\varphi,x)}^{-1}A_{(\varphi,x)},
\]
so $C_{(\varphi,x)}$ is a $2\times1$ matrix, which we identify with
a linear map $\mathbb{R}^{2}\rightarrow\mathbb{R}$. By Proposition
\ref{prop:iterated-entropy-formula-for-action} we have
\begin{eqnarray}
\frac{1}{n}H(\nu\bigdot\mu,\mathcal{D}_{n}) & \geq & \mathbb{E}_{0\leq i\leq n}\left(\frac{1}{k}H(C_{(\varphi,x)}\nu_{\varphi,i}*\mu_{x,i},\mathcal{D}_{i+k})\right)-o(1).\label{eq:9}
\end{eqnarray}
Suppose that for some $c=c(\varepsilon)>0$ it were true that
\begin{equation}
\mathbb{P}_{0\leq i\leq n}\left(\frac{1}{k}H(C_{(\varphi,x)}\nu_{\varphi,i}*\mu_{x,i},\mathcal{D}_{i+k})>\frac{1}{k}H(\mu_{x,i},\mathcal{D}_{i+k})+\delta'\right)>c.\label{eq:entropy-growth-on-average-1}
\end{equation}
Splitting the expectation in (\ref{eq:9}) by conditioning on the
event in (\ref{eq:entropy-growth-on-average-1}) and its complement,
using Lemma \ref{lem:entropy-monotonicity-underconvolution} to
control the expectation on the complement, and using Lemma \ref{lem:multiscale-formula-for-entropy},
we would have 
\begin{eqnarray*}
\frac{1}{n}H(\nu\bigdot\mu,\mathcal{D}_{n}) & \geq & \mathbb{E}_{0\leq i\leq n}\left(\frac{1}{k}H(\mu_{x,i},\mathcal{D}_{i+k})\right)+c\delta'-o(1)\\
 & = & \frac{1}{n}H(\mu,\mathcal{D}_{n})+c\delta'-o(1),
\end{eqnarray*}
as claimed. 

Now, by our choice of $\varepsilon'$ and $\delta'$, equation (\ref{eq:entropy-growth-on-average-1})
will follow if we show that
\begin{equation}
\mathbb{P}_{0\leq i\leq n}\left(\begin{array}{c}
\mu_{x,i}\mbox{ is }(1-\varepsilon',\delta',m)\mbox{-entropy porous at scales }\\
i\mbox{ to }i+k\mbox{, and }\frac{1}{k}H(C_{(\varphi,x)}\nu_{\varphi,i},\mathcal{D}_{i+k})>\varepsilon'
\end{array}\right)>c.\label{eq:7}
\end{equation}
This is the probability of an intersection of two events. The first,
involving $\mu_{x,i}$, can be dealt with using Lemma \ref{lem:bounded-component-entropy-passes-to-components}:
Indeed, by the hypothesis, if $m$ is large enough and $n$ suitably
large, then $\mu$ is $(1-\varepsilon,\delta,m)$-porous, and hence
$(1-\varepsilon',\delta,m)$-porous, at scales $0$ to $n$, so (assuming as we may that  $\delta<(\delta')^{2}/2$) Lemma \ref{lem:bounded-component-entropy-passes-to-components}
implies 
\[
\mathbb{P}_{0\leq i\leq n}\left(\mu_{x,i}\mbox{ is }(1-\varepsilon',\delta',m)\mbox{-entropy porous at scales }i\mbox{ to }i+k\right)=1-o(1).
\]
Thus, in order to prove (\ref{eq:7}), it remains to show that $\frac{1}{k}H(C_{(\varphi,x)}\nu_{\varphi,i},\mathcal{D}_{i+k})>\varepsilon'$
with probability bounded away from $0$, as $(\varphi,x)$ are chosen
according to $\nu\times\mu$ and $0\leq i\leq n$. Observe that
if the expression involved the entropy of $\nu_{\varphi,i}$ instead
of that of $C_{(\varphi,x)}\nu_{\varphi,i}$, we would be done,
because by Lemma \ref{lem:multiscale-formula-for-entropy} and our
hypothesis, 
\[
\mathbb{E}_{0\leq i\leq n}\left(\frac{1}{k}H(\nu_{\varphi,i},\mathcal{D}_{i+k})\right)=\frac{1}{n}H(\nu,\mathcal{D}_{n})-o(1)>\varepsilon-o(1),
\]
from which it follows that
\begin{equation}
\mathbb{P}_{0\leq i\leq n}\left(\frac{1}{k}H(\nu_{\varphi,i},\mathcal{D}_{i+k})>\frac{\varepsilon}{3}\right)>\frac{\varepsilon}{3}-o(1).\label{eq:8}
\end{equation}
The problem is that $C_{(\varphi,x)}$ is a linear map $\mathbb{R}^{2}\rightarrow\mathbb{R}$,
and has a 1-dimensional kernel, and if $\nu_{\varphi,i}$ happens
to be supported on (or close to) a translate of the kernel, then $C_{(\varphi,x)}\nu_{\varphi,i}$
is a Dirac measure (at least approximately), and has entropy (essentially)
equal zero no matter how large the entropy of $\nu_{\varphi,i}$
is.

The way to get around this problem is to note that the kernels of
these transformations are generally transverse to each other, and
intersect at a point; so if $\nu_{\varphi,i}$ has substantial
entropy it cannot be supported on or near $\ker B_{(\varphi,x)}^{-1}A_{(\varphi,x)}$
for too many values of $x$. Consequently, we shall show that conditioned
on $\varphi$ and $i$, with high $\mu$-probability over the choice
of $x$, $B_{(\varphi,x)}^{-1}A_{(\varphi,x)}\nu_{\varphi,i}$
must have at least a constant fraction of the entropy at scale $i+k$
as $\nu_{\varphi,i}$ itself. We prove this in the following sequence
of lemmas.

A map $f$ between metric spaces has bi-Lipschitz constant $c>0$
if $c^{-1}d(x,y)\leq d(f(x),f(y))\leq cd(x,y)$ for every $x,y$.
\begin{lemma}
\label{lem:separation}Let $g_{1},g_{2}:\mathbb{R}^{2}\rightarrow\mathbb{R}$
be such that the map $g(y)=(g_{1}(y),g_{2}(y))$ is bi-Lipschitz with
constant $c$. Then for any $\mu\in\mathcal{P}(\mathbb{R}^{2})$
and any $i$, some $j\in\{1,2\}$ satisfies 
\[
H(g_{j}\mu,\mathcal{D}_{i})>\frac{1}{2}H(\mu,\mathcal{D}_{i})-O(\log c).
\]
\end{lemma}
\begin{proof}
Since $g$ is bi-Lipschitz, by (\ref{eq:entropy-under-bi-lip-maps}),
\[
H(g\mu,\mathcal{D}_{i})=H(\mu,\mathcal{D}_{i})+O(\log c).
\]
Let $\pi_{j}$ be projection from $\mathbb{R}^{2}$ to the $j$-th
coordinate. Then $\mathcal{D}_{i}^{2}=\pi_{1}^{-1}\mathcal{D}_{i}\lor\pi_{2}^{-1}\mathcal{D}_{i}$,
so 
\begin{eqnarray*}
H(g\mu,\mathcal{D}_{i}) & = & H(g\mu,\pi_{1}^{-1}\mathcal{D}_{i}\lor\pi_{2}^{-1}\mathcal{D}_{i})\\
 & \leq & H(g\mu,\pi_{1}^{-1}\mathcal{D}_{i})+H(g\mu,\pi_{2}^{-1}\mathcal{D}_{i})\\
 & = & H(\pi_{1}g\mu,\mathcal{D}_{i})+H(\pi_{2}g\mu,\mathcal{D}_{i})\\
 & = & H(g_{1}\mu,\mathcal{D}_{i})+H(g_{2}\mu,\mathcal{D}_{i}),
\end{eqnarray*}
where in the last step we used the identity $\pi_{j}\circ g=g_{j}$.
Combining the last two equations gives the lemma.
\end{proof}
For $t>0$ let 
\[
\Sigma_{T}=\{(x,y)\,:\,|x-y|\geq t\}.
\]
Recall the definition of the matrix $A_{\varphi,x}$ preceding Proposition
\ref{prop:iterated-entropy-formula-for-action}.
\begin{lemma}
\label{lem:Lipschitz-constant-away-from-diagonal}Let $\varphi\in G$
and $x\neq y\in\mathbb{R}$. Then the map $g:\mathbb{R}^{2}\rightarrow\mathbb{R}$,
$g(z)=(A_{\varphi,x}z,A_{\varphi,y}z)$, is bi-Lipschitz, and for
$t>0$, for $(x,y)\in\Sigma_{t}\cap(\supp\mu)^{2}$ and $\varphi\in\supp\nu$,
its bi-Lipschitz constant is bounded uniformly by $O_{R}(1+t^{-1})$,
where $R$ is the smallest radius for which $\mu,\nu$ are supported
on the $R$-ball at the origin.\end{lemma}

Note that the first statement follows easily by observing that $\varphi\in G$
is determined by its action on any two points.
\begin{proof}
Suppose that $\varphi$ is represented by $(s,t)$ in coordinates,
so $f(\varphi,x)=\varphi(x)=e^{s}x+t$. A direct calculation yields
$A_{\varphi,x}=(\frac{\partial}{\partial s}f(\varphi,x),\frac{\partial}{\partial t}f(\varphi,x))=(e^{s}x,1)$,
hence the linear map $g$ in question is represented by the matrix
$\left(\begin{array}{cc}
e^{s}x & 1\\
e^{s}y & 1
\end{array}\right)$, which is invertible and has bi-Lipschitz with constant $O_{s}(1+|x-y|^{-1})$.
The second statement is immediate since $\Sigma_{t}\cap(\supp\mu)^{2}$
and $\supp\nu$ are compact. \end{proof}
\begin{lemma}
\label{lem:entropy-of-image}Let $\nu,\mu$ be as in Theorem \ref{thm:inverse-theorem-for-action}.
Let $h>0$, fix $\varphi,i$, and write $\theta=\nu_{\varphi,i}$.
Then, assuming that $\frac{1}{k}H(\theta,\mathcal{D}_{i+k})>h$, 
\[
\mu\left(x\in\mathbb{R}\,:\,\frac{1}{k}H(B_{(\varphi,x)}^{-1}A_{(\varphi,x)}\theta,\mathcal{D}_{i+k})>\frac{1}{3}h\right)=1-o_{h}(1)
\]
as $k\rightarrow\infty$, uniformly in $\varphi\in\supp\nu$ and
$i\in\mathbb{N}$. Furthermore if there are constants $a,\alpha>0$
such that $\mu(B_{r}(x))<ar^{\alpha}$ for all $x$ then the error
term is $o_{h,a,\alpha}(1/k^{d})$ for every $d$.\end{lemma}
\begin{proof}
By compactness $B_{(\varphi,x)}$ is bounded on the support of $\nu\times\mu$,
and scaling by a bounded constant changes entropy by $O(1)$; so,
after dividing by $k$, it changes by $o(1)$ (as $k\rightarrow\infty$).
Thus we may omit the factor $B_{(\varphi,x)}^{-1}$ in the statement.

Let $\rho>0$. Since we have assumed that $\mu$ is non-atomic, we
can fix $t>0$ such that $\mu(B_{t}(x))<\rho$ for all $x$.

Suppose for some $x'\in\supp\mu$ we have 
\[
\frac{1}{k}H(A_{(\varphi,x')}\theta,\mathcal{D}_{i+k})\leq\frac{1}{3}h
\]
(if no such $x'$ exists then we are done). Let $c$ denote the uniform
bound on the bi-Lipschitz constant associated to $t$ in Lemma \ref{lem:Lipschitz-constant-away-from-diagonal}.
By the previous two lemmas, if $(x,x')\in\Sigma_n|{t}\cap(\supp\mu)^{2}$
then necessarily 
\[
\frac{1}{k}H(A_{(\varphi,x)}\theta,\mathcal{D}_{i+k})\geq\frac{1}{2}h-O(\frac{\log c}{k})>\frac{1}{3}h,
\]
assuming $k$ large enough relative to $t$ (and hence $\rho$). This
implies that the event in the statement of the lemma contains $\mathbb{R}\setminus B_{t}(x')$
(up to a nullset). This set has $\mu$-measure at least $1-\rho$
by our choice of $t$. Thus we have shown that given $\rho$, if $k$
is large enough, then 
\begin{equation}
\mu\left(x\in\mathbb{R}\,:\,\frac{1}{k}H(A_{(\varphi,x)}\theta,\mathcal{D}_{i+k})>\frac{1}{3}h\right)>1-\rho.\label{eq:22}
\end{equation}

For the second statement, fix $d$ and $\rho=\rho_{k}=1/k^{d}$. By
assumption, $\mu(B_{r}(x))<ar^{\alpha}$ so in order for $t=t_{k}$
to satisfy $\mu(B_{t}(x))<\rho$ it suffices to take $t=O_{a}(\rho^{1/\alpha})=O_{a}(1/k^{d/\alpha})$.
Then by Lemma \ref{lem:Lipschitz-constant-away-from-diagonal} we
have $c=c_{k}=O(1+t^{-1})=O(k^{d/\alpha})$ and since the error term
$\log c_{k}/k$ in (\ref{eq:22}) tends to zero as $k\rightarrow\infty$,
the analysis above holds and the conclusion of the proposition is
valid with error term $O_{h,a,\alpha}(k^{d})$.
\end{proof}
We return to the proof of Theorem \ref{thm:inverse-theorem-for-action}.
Taking $h=\varepsilon/3$ in the last lemma and combining it with
equation (\ref{eq:8}), we find that for $k$ large enough, with probability
at least $\varepsilon/4$ (and hence probability at least $\varepsilon'$)
over our choice of $0\leq i\leq n$ and of $(\varphi,x)$ (chosen
with respect to $\nu\times\mu$), we will have $\frac{1}{k}H(B_{(\varphi,x)}^{-1}A_{(\varphi,x)}\theta,\mathcal{D}_{i+k})>\varepsilon/9>\varepsilon'$.
This completes the proof.

\subsection{\label{sub:Entropy-dimension}Entropy dimension}

Define the entropy dimension of $\mu\in\mathcal{P}(\mathbb{R})$
to be
\[
\edim\mu=\lim_{n\rightarrow\infty}\frac{1}{n}H(\mu,\mathcal{D}_{n}).
\]
if the limit exists, otherwise define the upper and lower entropy
dimensions $\uedim\mu,\ledim\mu$ by taking a limsup or liminf,
respectively. We also note that if $\mu$ is supported on a set $Y$
then by (\ref{eq:entropy-counting-bound}), $\ubdim Y\geq\uedim\mu$, where $\ubdim Y$ is the upper box
dimension, and a similar relation holds for lower entropy and box
dimensions 
\begin{theorem}
\label{thm:entropy-dimension-growth}For every $\varepsilon>0$ there
exists a $\delta=\delta(\varepsilon)>0$ such that the following holds: 

Let $\nu\in\mathcal{P}(G)$, $\mu\in\mathcal{P}(\mathbb{R})$
be compactly supported, and suppose that $\mu$ is non-atomic and
$(1-\varepsilon)$-entropy porous. Then 
\begin{eqnarray*}
\ledim\nu>\varepsilon & \implies & \ledim\nu\bigdot\mu>\ledim\mu+\delta\\
\ledim\nu>\varepsilon & \implies & \uedim\nu\bigdot\mu>\uedim\mu+\delta\\
\uedim\nu>\varepsilon & \implies & \uedim\nu\bigdot\mu>\ledim\mu+\delta.
\end{eqnarray*}

\end{theorem}
The proof is trivial from Theorem \ref{thm:inverse-theorem-for-action}
upon taking $n\rightarrow\infty$ and considering the definitions
of the upper and lower entropy dimensions. We leave the verification
to the reader.

Note that the case of convolutions $\nu*\mu$ for $\nu,\mu\in\mathbb{R}$
is contained in this theorem as a special case, since we can lift
$\nu$ to $\nu'\in\mathcal{P}(G)$ by identifying $t\in\mathbb{R}$
with the corresponding translation map $x\mapsto x+t$. Then $\nu*\mu=\nu'\bigdot\mu$,
and $\dim\nu=\dim\nu'$, so we get conditions for entropy-dimension
growth of Euclidean convolutions.

\section{\label{sec:Proof-of-main-Theorem}Proof of Theorem \ref{thm:main-for-measures}}

\subsection{Stationary measures}

Let $\Phi\subseteq\mathcal{S}$ be a compact set with $\dim\Phi>0$
and attractor $X$. Proving Theorem \ref{thm:main} requires us to
find suitable measures on $\Phi$ and $X$ to work with. For $\Phi$
we can take any measure $\nu$ of positive dimension, which exists
by Frostman's lemma (see e.g. \cite{Mattila1995}). There then exists a unique measure $\mu$ on
$\mathbb{R}$, called the $\nu$-stationary measure, satisfying 
\begin{equation}
\mu=\nu\bigdot\mu.\label{eq:Hutchison-measures}
\end{equation}
The existence and uniqueness of $\mu$ is proved by showing that $\tau\mapsto\nu.\tau$
is a contraction on $\mathcal{P}(\mathbb{R})$ when endowed with a
suitable metric. This is again the same argument as the one establishing
the existence of self-similar measures, and this is not surprising,
since self-similar measures are special cases of stationary ones:
when $\nu=\sum_{\varphi\in\Phi}p_{\varphi}\cdot\delta_{\varphi}$
is finitely supported, the relation (\ref{eq:Hutchison-measures})
becomes 
\begin{equation}
\mu=\sum_{\varphi\in\Phi}p_{\varphi}\cdot\varphi\mu,\label{eq:self-sim-measure}
\end{equation}
which is the definition of a self similar measure (as usual, $\varphi\mu=\mu\circ\varphi^{-1}$
is the push-forward of $\mu$ by $\varphi)$. We note that if a stationary
measure is not a single atom, then it is continuous (has no atoms).
The proof is standard and we omit it.

Recall the definition of entropy dimension from Section \ref{sub:Entropy-dimension}.
We show below (Proposition \ref{prop:entropy-dimension-exists}) that
if $\mu$ satisfies (\ref{eq:Hutchison-measures}) then its entropy
dimension exists.\footnote{One can also show that $\mu$ is exact-dimensional, but we do not
need this fact here.} We also show that it is $\edim\mu$-entropy porous (Proposition \ref{prop:mu-has-essentially-bounded-component-entropy}).
Then Theorem \ref{thm:entropy-dimension-growth} has the following
consequence:
\begin{theorem}
\label{thm:main-for-measures}Let $\mu\in\mathcal{P}(\mathbb{R})$
be a $\nu$-stationary measure for a compactly supported $\nu\in\mathcal{P}(\mathcal{S})$.
If $\uedim\nu>0$ then either $\mu$ is a Dirac measure, or $\edim\mu=1$.\end{theorem}
\begin{proof}
Suppose that $\mu\in\mathcal{P}(\mathbb{R})$ is not a Dirac measure.
Write $\alpha=\edim\mu$ and $\beta=\uedim\nu$. We assume that $\alpha<1$
and $\beta>0$, and wish to derive a contradiction. Set $\varepsilon=\frac{1}{2}\min\{\beta,1-\alpha\}>0$
and let $\delta=\delta(\varepsilon)$ be as in Theorem \ref{thm:inverse-theorem-for-action}.
Then $\mu$ is $(1-\varepsilon)$-entropy porous and continuous, so
by Theorem \ref{thm:inverse-theorem-for-action}, $\ledim\nu\bigdot\mu > \edim\mu+\delta$,
which is impossible.
\end{proof}
To complete the proof of Theorem \ref{thm:main} we must show that
$\edim\mu=1$ implies $\dim X=1$. This is simple: we have already
noted that $\edim\mu=1$ implies that $\bdim X=1$, and finally, this
implies $\dim X=1$ because $X$ has equal box and Hausdorff dimensions.
This last property is proven the same manner as for self-similar sets,
see e.g. \cite[Theorem 4]{Falconer1989}.

It remains for us to show that $\edim\mu$ exists and that it is entropy
porous. We do this in the couple of sections.

\subsection{\label{sub:Cylinder-decomposition-and-entropy-dimension}Cylinder
decomposition of stationary measures and entropy dimension}

Let  $\Phi\subseteq\mathcal{S}$
be compact, let
$r_{0}=\min_{\varphi\in\Phi}\left\Vert \varphi\right\Vert $ and
given $n\in\mathbb{N}$ let $\Phi_{n}$ denote the set \[
\Phi_n \;=\; \left\lbrace (\varphi_1,\ldots,\varphi_k)\in\bigcup_{\ell=1}^\infty \Phi^\ell\,:\,r_{0}2^{-n} \leq\left\Vert \varphi_1\ldots \varphi_k\right\Vert <2^{-n}\right\rbrace
\]

Suppose that $\Phi=\{\varphi_{1},\ldots,\varphi_{m}\}$ is finite, so also $\Phi_n$ is finite, and that $\mu=\sum_{i=1}^m p_{i}\cdot\varphi_{i}\mu$ is a self-similar measure. For $\varphi_{i_{1}},\ldots,\varphi_{i_{k}}\in\Phi$ write $\varphi_{i_{1}\ldots i_{k}}=\varphi_{i_{1}}\varphi_{i_{2}}\ldots\varphi_{i_{k}}$
and $(p_i)_{i=1}^m$ write  $p_{i_{1}\ldots i_{k}}=p_{i_{1}}p_{i_{2}}\ldots p_{i_{k}}$. Then one can iterate the definition of $\mu$ to get
\begin{equation}
\mu=\sum_{(\varphi_{i_{1}}\ldots,\varphi_{i_{k}})\in\Phi_{n}}p_{i_{1}\ldots i_{k}}\cdot\varphi_{i_{1}\ldots i_{k}}\mu.\label{eq:discrete-cylinder-decomposition}
\end{equation}
This ``decomposes'' $\mu$ into finitely many images of itself,
each by a map which contracts by roughly $2^{-n}$.

Now let $\Phi\subseteq\mathcal{S}$ be a general compact set, $\nu\in\mathcal{P}(\mathcal{S})$ a compactly supported, and 
$\mu$ a $\nu$-stationary, $\nu\bigdot\mu=\mu$. We want to have a
similar representation of $\mu$, but now instead of a sum we will
have an integral, the family $\Phi_{n}$ generally being uncountable,
and a suitable measure replacing the weights $p_{i_{1}\ldots i_{k}}$
in the sum. The way to do this is to consider the Markov chain obtained
by repeatedly applying to $\mu$ a random map, chosen according to
$\nu$. Indeed the relation $\mu=\nu\bigdot\mu$ just means that, if
$\varphi$ denotes a random similarity chosen according to $\nu$,
then
\[
\mu=\mathbb{E}(\varphi\mu).
\]
Thus let $(\varphi_{i})_{i=1}^{\infty}$ be an independent sequence
of similarities with common distribution $\nu$ and consider the measure-valued
random process 
\[
\mu_{n}=\varphi_{1}\varphi_{2}\ldots\varphi_{n}\mu.
\]
This is a martingale with respect to the filtration $\mathcal{F}_{n}=\theta(\varphi_{1},\ldots,\varphi_{n})$,
since, writing $\Omega$ for the sample space of the process and $\mu_{n}^{\omega}$
to indicate the dependence on $\omega\in\Omega$, 
\begin{eqnarray*}
\mathbb{E}(\mu_{n+1}|\mathcal{F}_{n})(\omega) & = & \mathbb{E}(\varphi_{1}(\omega)\cdot\ldots\cdot\varphi_{n}(\omega)\cdot\varphi_{n+1}\mu)\\
 & = & (\varphi_{1}(\omega)\cdot\ldots\cdot\varphi_{n}(\omega))\mathbb{E}(\varphi_{n+1}\mu)\\
 & = & (\varphi_{1}(\omega)\cdot\ldots\cdot\varphi_{n}(\omega))\mu\\
 & = & \mu_{n}^{\omega}
\end{eqnarray*}
(in the second line we use the easy fact that integrating measures
commutes with pushing them forward). 

Recall that a random variable $\tau$ is a stopping time for $(\mathcal{F}_{n})$
if the event $\{\tau\leq k\}$ belongs to $\mathcal{F}_{k}$ for all
$k\in\mathbb{N}$. Given a bounded stopping time, Doob's optional
stopping theorem \cite[Theorem 7.12]{Kallenberg2002} asserts that\footnote{To derive this from the sampling theorem for real-valued random variables,
note that we need to show that $\int f\,d\mathbb{E}(\mu_{\tau})=\int fd\mu$
for all bounded functions $f$, and this follows since $\xi_{n}=\int fd\mu_{n}$
is easily seen to be a martingale for $(\mathcal{F}_{n})$, and by
Fubini $\int f\,d\mathbb{E}(\mu_{\tau})=\mathbb{E}(\int fd\mu_{\tau})=\mathbb{E}(\xi_{\tau})=\mathbb{E}(\xi_{0})=\int fd\mu$,
where we used the real-valued optional stopping theorem in the second
to last equality. } 
\[
\mathbb{E}(\mu_{\tau})=\mathbb{E}(\mu_{0})=\mu.
\]

We apply this to the stopping time 
\begin{equation}
\tau_{n}=\min\{k\in\mathbb{N}\,:\,\left\Vert \varphi_{1}\ldots\varphi_{k}\right\Vert <2^{-n}\}.\label{eq:tau-n}
\end{equation}
Since $\supp\nu$ is compact, there exist $0<r_{0}<r_{1}<1$ such
that $r_{0}\leq\left\Vert \varphi\right\Vert \leq r_{1}$ for all
$\varphi\in\Phi$, which implies that $2^{-n}r_{0}\leq\left\Vert \varphi_{1}\ldots\varphi_{\tau_{n}}\right\Vert <2^{-n}$,
and also that $\tau_{n}\leq n/\log(1/r_{1})$, so $\tau_{n}$ is bounded.
Therefore the identity 
\[
\mu=\mathbb{E}(\mu_{\tau})=\mathbb{E}(\varphi_{1}\ldots\varphi_{\tau}\mu)
\]
is the desired analog of (\ref{eq:discrete-cylinder-decomposition}).
\begin{proposition}
\label{prop:entropy-dimension-exists}$\edim\mu=\lim_{n\rightarrow\infty}\frac{1}{n}H(\mu,\mathcal{D}_{n})$
exists.\end{proposition}
\begin{proof}
For any $m,n$, by (\ref{eq:entorpy-under-similarity}) and by the
fact that $\varphi_{1}\ldots\varphi_{\tau}$ contracts by $2^{-(n+O(1))}$,
we see that $\mu_{\tau_{n}}$ is supported on a set of diameter $2^{-n+O(1)}$.
Therefore 
\[
H(\mu_{\tau_{n}},\mathcal{D}_{n+m})=H(\mu,\mathcal{D}_{m})+O(1).
\]
and for the same reason, by (\ref{eq:conditional-entropy-combinatorial-bound}),
\[
H(\mu_{\tau_{n}},\mathcal{D}_{n+m}|D_{n})=H(\mu_{\tau_{n}},\mathcal{D}_{n+m})+O(1).
\]

Write $a_{n}=H(\mu,\mathcal{D}_{n})$. Then by concavity of conditional
entropy and the discussion above, 
\begin{eqnarray*}
a_{m+n} & = & H(\mu,\mathcal{D}_{m+n})\\
 & = & H(\mu,\mathcal{D}_{n})+H(\mu,\mathcal{D}_{m+n}|\mathcal{D}_{n})\\
 & = & H(\mu,\mathcal{D}_{n})+H(\mathbb{E}(\mu_{\tau_{n}}),\mathcal{D}_{m+n}|\mathcal{D}_{n})\\
 & \geq & H(\mu,\mathcal{D}_{n})+\mathbb{E}\left(H(\mu_{\tau_{n}},\mathcal{D}_{m+n}|\mathcal{D}_{n})\right)\\
 & = & H(\mu,\mathcal{D}_{n})+\mathbb{E}\left(H(\mu,\mathcal{D}_{m})+O(1)\right)\\
 & = & a_{m}+a_{n}+O(1).
\end{eqnarray*}
It follows that up to an $O(1)$ error $(a_{n})$ is super-additive,
so $\lim_{n\rightarrow\infty}\frac{1}{n}a_{n}$ exists, as desired.\footnote{One way to see this is by adapting the proof of the classical Fekete
lemma. Alternatively consider $b_{n}=a_{n}-\sqrt{n}$, which after
dividing by $n$ has the same asymptotics as $a_{n}$ , but satisfies
$b_{m+n}\geq b_{m}+b_{n}$ for all $m,n$ large enough, so that Fekete's
lemma applies to it. }
\end{proof}

\subsection{\label{sub:Entropy-porosity-of-stationary-measures}Entropy porosity
of stationary measures}

Returning to our stationary measures, our next goal is to show that
they are entropy-porous. The argument is essentially the same as in
\cite[Section 5.1]{Hochman2014}, with some additional minor complications
due to continuity of $\nu$.

Let $\mu=\nu\bigdot\mu$ be a stationary measure for a compactly supported
$\nu\in\mathcal{P}(\mathcal{S})$, and assume $\mu$ is not a Dirac
measure. By a change of coordinates $x\mapsto2^{-N}(x+k)$ for suitable
choice of $N,k\in\mathbb{N}$, we may assume that $\mu$ is supported
on $[0,1/2)$. Write 
\[
\alpha=\edim\mu.
\]
Our goal is to prove the following:
\begin{proposition}
\label{prop:mu-has-essentially-bounded-component-entropy} For every
$\varepsilon>0$ and $m>m(\varepsilon)$, for all large enough $n$,
\[
\mathbb{P}_{0\leq i\leq n}\left(|\frac{1}{m}H(\mu_{x,i},\mathcal{D}_{i+m})-\alpha|<\varepsilon\right)>1-\varepsilon.
\]
In particular, $\mu$ is $\alpha$-entropy porous, and satisfies the
conclusion of Lemma \ref{lem:bounded-component-entropy-passes-to-components}.
\end{proposition}
To prove this we need only prove that for every $\varepsilon>0$,
$m>m(\varepsilon)$ and all $n$,
\begin{equation}
\mathbb{P}_{0\leq i\leq n}\left(\frac{1}{m}H(\mu_{x,i},\mathcal{D}_{i+m})>\alpha-\varepsilon\right)>1-\varepsilon.\label{eq:5}
\end{equation}
Indeed, by Lemma \ref{lem:multiscale-formula-for-entropy} and the
fact that $\frac{1}{n}H(\mu,\mathcal{D}_{n})\rightarrow\alpha$, for
large enough $n$, 
\begin{equation}
\mathbb{E}_{0\leq i\leq n}(\frac{1}{m}H(\mu_{x,i},\mathcal{D}_{i+m}))\leq\alpha+\varepsilon.\label{eq:4}
\end{equation}
This is an average of a non-negative quantity which, by (\ref{eq:5}),
with probability $1-\varepsilon$ is not more than $2\varepsilon$
less than its mean, so 
\[
\mathbb{P}_{0\leq i\leq n}\left(\frac{1}{m}H(\mu_{x,i},\mathcal{D}_{i+m})>\alpha+\sqrt{2\varepsilon}\right)<\sqrt{2\varepsilon}.
\]
Starting from $\varepsilon^{2}/8$ instead of $\varepsilon$, and
combining with (\ref{eq:5}), this proves the proposition.

We turn to the proof of (\ref{eq:5}). Let $\delta>0$ be a parameter
to be determined later. Since $\mu$ is not a Dirac measure it is
continuous (has no atoms), so there is a $\rho>0$ such that $\mu(B_{\rho}(x))<\delta$
for all $x\in\mathbb{R}$ (here and throughout, balls are open). We
can assume that $\rho<\frac{1}{4}$. 

Let $\varphi_{1},\varphi_{2},\ldots$ be an i.i.d. sequence with marginal
$\nu$, defined on some sample space $\Omega$. Let $\tau_{i}$ be
the stopping time defined in (\ref{eq:tau-n}). 

Denote $r_{0}=\inf\{\left\Vert \varphi\right\Vert \,:\,\varphi\in\supp\nu\}$.
Fix $i$ and let $V_{i}\subseteq\mathbb{R}$ denote the set of points
$x$ whose distance from $\mathbb{Z}/2^{i}$ is less than $2^{-i}\rho r_{0}$,
that is, $V_{i}=\bigcup_{k\in\mathbb{Z}}B_{2^{-i}\rho r_{0}}(k/2^{i})$. 
\begin{lemma}
$\mu(V_{i})<\delta$. \end{lemma}
\begin{proof}
Since $\mu=\mathbb{E}(\varphi_{1}\ldots\varphi_{\tau_{i}}\mu)$, it
is enough to show that $(\varphi_{1}\ldots\varphi_{\tau_{i}}\mu)(V_{i})<\delta$
a.s. over the choice of the maps. Writing $r_{i}=\left\Vert \varphi_{1}\ldots\varphi_{\tau_{i}}\right\Vert $,
for some $t_{i}\in\mathbb{R}$ we have 
\begin{eqnarray*}
(\varphi_{1}\ldots\varphi_{\tau_{i}}\mu)(V_{i}) & = & \mu((\varphi_{1}\ldots\varphi_{\tau_{i}})^{-1}V_{i})\\
 & = & \mu(\frac{1}{r_{i}}V_{i}+t_{i}).
\end{eqnarray*}
But by definition of $\tau_{i}$ we have $2^{-i}r_{0}<r_{i}\leq2^{-i}$,
so 
\begin{eqnarray*}
\frac{1}{r_{i}}V_{i}+t_{i} & = & \bigcup_{k\in\mathbb{Z}}B_{2^{-i}\rho r_{0}/r_{i}}(k/(r_{i}2^{i}))+t_{i}\\
 & \subseteq & \bigcup_{k\in\mathbb{Z}}B_{\rho}(t_{i}+k/(r_{i}2^{i})).
\end{eqnarray*}
On the other hand, $\{t_{i}+k/(r_{i}2^{i})\}_{k\in\mathbb{Z}}$ is
a periodic sequence with gap size at least $1$, and since $\rho<1/4$
and $\mu$ is supported on a set of diameter $1/2$, at most one of
the balls $B_{\rho}(t_{i}+k/(r2_{i}^{i}))$ intersects the support
of $\mu$. The $\mu$-mass of this ball is less than $\delta$ by
our choice of $\rho$, and the claim follows.\end{proof}
\begin{lemma}
\label{lem:most-cells-mostly-covered-by-cyls}$\mu(x\,:\,\mu(\mathcal{D}_{i}(x)\cap V_{i})<\sqrt{\delta}\mu(\mathcal{D}_{i}(x)))>1-\sqrt{\delta}$.\end{lemma}
\begin{proof}
Elementary, using $\mu(V_{i})<\delta$.
\end{proof}
Let $\ell\in\mathbb{N}$ be large enough that the diameter of $\supp\mu$
is less than $2^{\ell}\rho r_{0}$. Assume that $D\in\mathcal{D}_{i}$
and $\mu(D)>0$. Then $\mu=\mathbb{E}(\varphi_{1}\ldots\varphi_{\tau_{i+\ell}}\mu)$
implies
\[
\mu|_{D}=\mathbb{E}\left((\varphi_{1}\ldots\varphi_{\tau_{i+\ell}}\mu)|_{D}\right).
\]
Let $\mathcal{A}_{D}$ denote the event that $(\varphi_{1}\ldots\varphi_{\tau_{i+\ell}}\mu)(D)=1$
and $\mathcal{B}_{D}$ the event that $0<(\varphi_{1}\ldots\varphi_{\tau_{i+\ell}}\mu)(D)<1$.
Then we have
\begin{align}
  \mu|_{D}\;\;= \;&\mathbb{P}(\mathcal{A}_{D})\cdot\mathbb{E}\left((\varphi_{1}\ldots\varphi_{\tau_{i+\ell}}\mu)|_{D}\biggl|\mathcal{A}_{D}\right) \nonumber\\
   &\;+\mathbb{P}(\mathcal{B}_{D})\cdot\mathbb{E}\left((\varphi_{1}\ldots\varphi_{\tau_{i+\ell}}\mu)|_{D}\biggl|\mathcal{B}_{D}\right)
\end{align}
(the missing term, where the expectation is conditioned on the complement of   $\mathcal{A}_{D}\cup\mathcal{B}_{D}$,
is zero). Dividing the equation by $\mu(D)$, and dividing and multiplying
each integrand by $(\varphi_{1}\ldots\varphi_{\tau_{i+\ell}}\mu)(D)$
and using the fact that this is $1$ on $\mathcal{A}_{D}$, we obtain
\begin{align}
\mu_{D}  = &\; \frac{\mathbb{P}(\mathcal{A}_{D})}{\mu(D)}\cdot\mathbb{E}\left((\varphi_{1}\ldots\varphi_{\tau_{i+\ell}}\mu)_{D}\left|\mathcal{A}_{D}\right.\right)+\nonumber \\
 &  \quad+\;\frac{\mathbb{P}(\mathcal{B}_{D})}{\mu(D)}\cdot\mathbb{E}\left((\varphi_{1}\ldots\varphi_{\tau_{i+\ell}}\mu)(D)\cdot(\varphi_{1}\ldots\varphi_{\tau_{i+\ell}}\mu)_{D}\left|\mathcal{B}_{D}\right.\right).\label{eq:2}
\end{align}
Evaluating this measure-valued equation on $D$ shows that 
\begin{equation}
\frac{\mathbb{P}(\mathcal{A}_{D})}{\mu(D)}+\frac{\mathbb{P}(\mathcal{B}_{D})}{\mu(D)}\cdot\mathbb{E}\left((\varphi_{1}\ldots\varphi_{\tau_{i+\ell}}\mu)(D)\left|\mathcal{B}_{D}\right.\right)=1.\label{eq:3}
\end{equation}

\begin{lemma}
If $D\in\mathcal{D}_{i}$ and $\mu(D\cap V_{i})<\sqrt{\delta}\mu(D)$
then $\mathbb{P}(\mathcal{A}_{D})/\mu(D)>1-\sqrt{\delta}$.\end{lemma}
\begin{proof}
Suppose that $0<(\varphi_{1}\ldots\varphi_{\tau_{i+\ell}}\mu)(D)<1$.
Then $\varphi_{1}\ldots\varphi_{\tau_{i+\ell}}\mu$ gives positive
mass to both $D$ and $\mathbb{R}\setminus D$. On the other hand
the diameter of this measure is at most $2^{-(i+\ell)}$ times the
diameter of $\supp\mu$, which by choice of $\ell$ is at most $\rho2^{-i}$,
so $\varphi_{1}\ldots\varphi_{\tau_{i+\ell}}\mu$ must be supported
within $\rho2^{-i}$ of $\partial D$, and hence it is supported on
$V_{i}$. We have found that on the event $\mathcal{B}_{D}$, if $(\varphi_{1}\ldots\varphi_{\tau_{i+\ell}}\mu)(D)>0$
then $(\varphi_{1}\ldots\varphi_{\tau_{i+\ell}}\mu)(V_{i})=1$, and
therefore also $(\varphi_{1}\ldots\varphi_{\tau_{i+\ell}}\mu)_{D}(V_{i})=1$.
Consequently, by our hypothesis and (\ref{eq:2}),
\begin{eqnarray*}
\sqrt{\delta} & > & \mu_{D}(V_{i})\\
 & \geq & \frac{\mathbb{P}(\mathcal{B}_{D})}{\mu(D)}\cdot\mathbb{E}\left((\varphi_{1}\ldots\varphi_{\tau_{i+\ell}}\mu)(D)\cdot(\varphi_{1}\ldots\varphi_{\tau_{i+\ell}}\mu)_{D}(V_{i})\left|\mathcal{B}_{D}\right.\right)\\
 & = & \frac{\mathbb{P}(\mathcal{B}_{D})}{\mu(D)}\cdot\mathbb{E}\left((\varphi_{1}\ldots\varphi_{\tau_{i+\ell}}\mu)(D)\left|\mathcal{B}_{D}\right.\right).
\end{eqnarray*}
The claim follows using (\ref{eq:3}).
\end{proof}
We now prove (\ref{eq:5}), proving Proposition \ref{prop:mu-has-essentially-bounded-component-entropy}.
Let $\varepsilon>0$, and continue with the previous notation, eventually
taking $\delta$ small relative to $\varepsilon$, and $m$ large
relative to $\varepsilon,\delta$ (and hence relative to $\rho$ and
$\ell$, since they are determined by $\delta$).

Suppose that $D\in\mathcal{D}_{i}$ and $\mu(D\cap V_{i})<\sqrt{\delta}\mu(D)$.
By (\ref{eq:2}) we can write 
\[
\mu_{D}=\frac{\mathbb{P}(\mathcal{A}_{D})}{\mu(D)}\mathbb{E}((\varphi_{1}\ldots\varphi_{\tau_{i+\ell}}\mu)_{D}|\mathcal{A}_{D})+(1-\frac{\mathbb{P}(\mathcal{A}_{D})}{\mu(D)})\nu
\]
for some probability measure $\nu$. By concavity of entropy and
the last lemma, 
\begin{eqnarray*}
\frac{1}{m}H(\mu_{D},\mathcal{D}_{i+m}) & \geq & \frac{\mathbb{P}(\mathcal{A}_{D})}{\mu(D)}\cdot\frac{1}{m}H\left(\mathbb{E}((\varphi_{1}\ldots\varphi_{\tau_{i+\ell}}\mu)_{D}|\mathcal{A}_{D}),\mathcal{D}_{i+m}\right)\\
 & \geq & (1-\sqrt{\delta})\mathbb{E}\left(\frac{1}{m}H((\varphi_{1}\ldots\varphi_{\tau_{i+\ell}}\mu)_{D},\mathcal{D}_{i+m})\left|\mathcal{A}_{D}\right.\right).
\end{eqnarray*}
Conditioned on the event $\mathcal{A}_{D}$ we have $(\varphi_{1}\ldots\varphi_{\tau_{i+\ell}}\mu)_{D}=\varphi_{1}\ldots\varphi_{\tau_{i+\ell}}\mu$,
and since $\varphi_{1}\ldots\varphi_{\tau_{i+\ell}}$ contracts by
at most $2^{-(i+\ell)}r_{0}$, we have
\begin{eqnarray*}
\frac{1}{m}H(\varphi_{1}\ldots\varphi_{\tau_{i+\ell}}\mu,\mathcal{D}_{i+m}) & = & \frac{1}{m}H(\mu,\mathcal{D}_{m})+O_{r_{0},\ell}(\frac{1}{m}).
\end{eqnarray*}
Combined with the previous inequality we obtain
\[
\frac{1}{m}H(\mu_{D},\mathcal{D}_{i+m})\geq\frac{(1-\sqrt{\delta})}{m}\cdot H(\mu,\mathcal{D}_{m})+O_{r_{0},\ell}(\frac{1}{m})\geq\alpha-\varepsilon,
\]
assuming $\delta$ is small and $m$ large.

The analysis above holds for $D\in\mathcal{D}_{i}$ such that $\mu(D\cap V_{i})<\sqrt{\delta}\mu(D)$.
By Lemma \ref{lem:most-cells-mostly-covered-by-cyls}, and assuming
as we may that $\delta<\varepsilon^{2}$ and $m$ is large enough,
this implies the proposition.

\section{\label{sec:Growth-of-Hausdorff-dim}Growth of Hausdorff dimension
under convolution}

So far we have analyzed the growth of entropy at fixed small scales,
which in the limit leads to results for entropy dimension. We now
turn the growth of the Hausdorff dimension of measures. Technically,
involves replacing the ``global'' distribution of components $\mathbb{P}_{n}^{\eta}$,
in which $\eta_{x,i}$ is selected by randomizing both $x$ and $i$,
with ``pointwise'' distributions of components, where $x$ is fixed
and we average only over the scales. This requires us to modify some
of the definitions and slightly strengthen the hypotheses. It also
calls for some additional analysis, based to a large extent on the
local entropy averages method.

\subsection{\label{sub:Hausdorff-dimension-and-pointwise-dim}Hausdorff and pointwise
dimension }

To start off, recall that the (lower) Hausdorff dimension of a measure
$\eta\in\mathcal{P}(\mathbb{R})$ is given by 
\[
\ldim\eta=\inf\{\dim E\,:\,\eta(E)>0\},
\]
as $E$ ranges over Borel sets. Unlike entropy dimension, which averages
the behavior of a measure over space, Hausdorff dimension is determined
by the pointwise behavior of a measure. Indeed, define the (lower,
dyadic) pointwise dimension of $\eta$ at $x$ to be 
\[
\ldim(\eta,x)=\liminf_{n\rightarrow\infty}-\frac{\log\eta(\mathcal{D}_{n}(x))}{n}.
\]
(one may take the limit along integer or continuous parameter $n$).
Then
\[
\ldim\eta=\essinf_{x\sim\eta}d(\eta,x).
\]
It is elementary that if $n_{i}\rightarrow\infty$ and $n_{i+1}/n_{i}\rightarrow1$
then in the definition of $\ldim(\eta,x)$ we can take the limit along
$n_{i}$. For reasons which will become apparent later we will want
to take advantage of this freedom.

We mention a basic stability property of the local dimension:
\begin{lemma}
\label{lem:hausdorff-dim-stable-under-conditioning}If $\eta\ll\theta$
are probability measures on $\mathbb{R}$ then $\ldim(\eta,x)=\ldim(\theta,x)$
for $\eta$-a.e. $x$.
\end{lemma}
This is a consequence of the martingale convergence theorem, according
to which for $\eta,\theta$ as in the lemma, $\frac{\eta(\mathcal{D}_{n}(x))}{\theta(\mathcal{D}_{n}(x))}\rightarrow\frac{d\eta}{d\theta}(x)\in(0,\infty)$
at $\eta$-a.e. point $x$.

\subsection{\label{sub:Local-entropy-averages}Local entropy averages}

The connection of pointwise dimension and entropy is via the so-called
local entropy averages method, introduced in \cite{HochmanShmerkin2012}.
This can be regarded as a pointwise analog of Lemmas \ref{lem:multiscale-formula-for-entropy}
and \ref{lem:multiscale-formula-for-entropy-of-additive-convolution}.
We give a version of the lemma along a sparse sequence of scales,
specifically, of power growth. Let $[\cdot]$ denote the integer value
function. 
\begin{lemma}
\label{lem:local-entropy-averages}Let $\tau>0$ and let $n_{i}=[i^{1+\tau}]$.
Then for any $\eta\in\mathcal{P}(\mathbb{R}^{d})$ and $\eta$-a.e.
$x$,
\begin{equation}
\ldim(\eta,x)\geq\liminf_{k\rightarrow\infty}\frac{1}{k}\sum_{i=0}^{k-1}\frac{1}{n_{i+1}-n_{i}}H(\eta_{x,n_{i}},\mathcal{D}_{n_{i+1}})-\tau,\label{eq:local-entropy-averages}
\end{equation}
and if $\theta\in\mathcal{P}(\mathcal{S})$ and $\eta\in\mathcal{P}(\mathbb{R})$,
then for $\theta\times\eta$-a.e. $(\varphi,x)$ and $y=\varphi(x)$,
\begin{equation}
\ldim(\theta\bigdot\eta,y)\geq\liminf_{k\rightarrow\infty}\frac{1}{k}\sum_{i=0}^{k-1}\frac{1}{n_{i+1}-n_{i}}H(\theta_{\varphi,n_{i}}\bigdot\eta_{x,n_{i}},\mathcal{D}_{n_{i+1}})-\tau.\label{eq:local-entropy-averages-convolution}
\end{equation}
\end{lemma}
\begin{proof}
We start with the first statement. Clearly $2^{-n_{i}}=2^{-n_{i-1}(1+o(1))}$,
so $\ldim(\eta,x)=-\liminf\frac{1}{n_{k}}\log\mu(\mathcal{D}_{n_{k}}(x))$.
Set $w_{k,i}=(n_{i}-n_{i-1})/n_{k}$, so $(w_{k,1},\ldots,w_{k,k})$
is a probability vector. From 
\[
\eta(\mathcal{D}_{n_{k}}(x))=\sum_{i=0}^{k}\log\frac{\mu(\mathcal{D}_{n_{i}}(x))}{\mu(\mathcal{D}_{n_{i-1}}(x))}
\]
we find that 
\begin{eqnarray*}
\ldim(\eta,x) & = & -\liminf_{k\rightarrow\infty}\frac{1}{n_{k}}\sum_{i=1}^{k}\log\frac{\mu(\mathcal{D}_{n_{i}}(x))}{\mu(\mathcal{D}_{n_{i-1}}(x))}\\
 & = & \liminf_{k\rightarrow\infty}\sum_{i=1}^{k}w_{k,i}\cdot\left(-\frac{1}{n_{i}-n_{i-1}}\log\frac{\mu(\mathcal{D}_{n_{i}}(x))}{\mu(\mathcal{D}_{n_{i-1}}(x))}\right)
\end{eqnarray*}
By a variation on the law of large numbers for one-sided bounded uncorrelated
$L^{2}$ random variables\footnote{Here is a proof sketch: Let $(X_{i})$ be a martingale with $\mathbb{E}X_{i}=0$,
$\mathbb{E}(X_{i}^{2})\leq a$, and $X_{i}\geq-b$ for some constants
$a,b>0$. Let $w_{k,i}$ be as before, write $S_{k}=\sum_{i=1}^{k}w_{k,i}X_{i}$.
We claim that $\liminf_{k}S_{k}\geq0$ a.s. Consider first the subsequence
$S_{k^{2}}$. Using $w_{k,i}=(1+\tau+o(1))k^{-(1+\tau)}i^{\tau}$
and $\mathbb{E}(X_{i}X_{j})=0$ for $i\neq j$, we have 
\[
\mathbb{E}((S_{k^{2}})^{2})=\sum_{i=1}^{k^{2}}w_{k^{2},i}^{2}\mathbb{E}(X_{i}^{2})=O(k^{-4(1+\tau)}\sum_{i=1}^{k^{2}}i^{2\tau})=O(k^{-2})
\]
Hence by Markov's inequality $\sum\mathbb{P}(S_{k^{2}}>\varepsilon)<\infty$,
and by Borel-Cantelli, $S_{k^{2}}\rightarrow0$ a.s. We now interpolate:
for $k^{2}\leq\ell<(k+1)^{2}$ and using $w_{\ell,i}=(\frac{\ell}{k^{2}})^{1+\tau}w_{k^{2},i}$
and $X_{i}\geq-b$ we have 
\begin{eqnarray*}
S_{\ell} & = & \sum_{i=1}^{k^{2}}w_{\ell,i}X_{i}+\sum_{i=k^{2}+1}^{\ell}w_{\ell,i}X_{i}\\
 & \geq & (\frac{\ell}{k^{2}})^{1+\tau}S_{k^{2}}-\sum_{i=k^{2}}^{\ell}w_{\ell,i}b\\
 & = & (1+o_{\tau}(1))S_{k^{2}}-o_{b,\tau}(1),
\end{eqnarray*}
from which the claim follows.} shows that $\eta$-a.e. $x$ satisfies 
\[
\lim_{k\rightarrow\infty}\frac{1}{k}\sum_{i=1}^{k}\left(-\frac{1}{n_{i}-n_{i-1}}H(\eta,\mathcal{D}_{n_{i}}|\mathcal{D}_{n_{i}})-\frac{1}{n_{i}-n_{i-1}}\log\frac{\mu(\mathcal{D}_{n_{i}}(x))}{\mu(\mathcal{D}_{n_{i-1}}(x))}\right)\geq0,
\]
so
\[
\ldim(\eta,x)\geq\liminf_{k\rightarrow\infty}\sum_{i=1}^{k}w_{k,i}\cdot\frac{1}{n_{i}-n_{i-1}}H(\eta,\mathcal{D}_{n_{i}}|\mathcal{D}_{n_{i}}).
\]
Finally, writing $a_{i}=H(\eta,\mathcal{D}_{n_{i}}|\mathcal{D}_{n_{i-1}})/(n_{i}-n_{i-1})$,
the proof is completed by showing that $\sum_{i=1}^{k}w_{k,i}a_{i}=\frac{1}{k}\sum_{i=1}^{k}a_{i}-\tau-o(1)$
as $k\rightarrow\infty$. Indeed, let 
\begin{eqnarray*}
E_{k} & = & \{(i,j)\in\mathbb{Z}^{2}\,:\,1\leq i\leq k\,,\,1\leq j\leq(1+\tau)k^{\tau}\}\\
F_{k} & = & \{(i,j)\in\mathbb{Z}^{2}\,:\,1\leq i\leq k,,,1\leq j\leq i^{1+\tau}-(i-1)^{1+\tau}\}\\
 & \subseteq & E_{k}
\end{eqnarray*}
Evidently, 
\begin{eqnarray*}
\frac{1}{k}\sum_{i=1}^{k}a_{i} & = & \frac{1}{|E_{k}|}\sum_{(i,j)\in E_{k}}a_{i}\\
\sum_{i=1}^{k}w_{k,i}a_{i} & = & \frac{1}{|F_{k}|}\sum_{(i,j)\in F_{k}}a_{i}
\end{eqnarray*}
An elementary calculatoin also shows that $|E_{k}|/|F_{k}|=1+\tau+o(1)$.
This, together with $|a_{i}|\leq1$, implies that 
\begin{eqnarray*}
\sum_{i=1}^{k}w_{k,i}a_{i} & = & \frac{1}{|F_{k}|}\sum_{(i,j)\in E_{k}}a_{i}-\frac{1}{|F_{k}|}\sum_{(i,j)\in E_{k}\setminus F_{k}}a_{i}-o(1)\\
 & \geq & \frac{|F_{k}|}{|E_{k}|}\cdot\frac{1}{|E_{k}|}\sum_{(i,j)\in E_{k}}a_{i}-\frac{1}{|F_{k}|}(|E_{k}|-|F_{k}|)-o(1)\\
 & = & (1+\tau)\frac{1}{k}\sum_{i=1}^{k}a_{i}-(1+\tau-1)-o(1).\\
 & \geq & \frac{1}{k}\sum_{i=1}^{k}a_{i}-\tau-o(1).
\end{eqnarray*}
as desired.

The second part of the lemma is a similar adaptation of the local
entropy averages lemma to the action setting, similar to the projection
case in \cite{HochmanShmerkin2012}. We omit the details.
\end{proof}
We need a variant for convolutions in the action setting, which may
be regarded as a pointwise analog of Lemma \ref{lem:multiscale-formula-for-entropy-of-action-convolution}.
To control the error term in the linearization, we use the fact that
 $n_{i}=[i^{1+\tau}]$ satisfies $n_{i+1}-n_{i}\rightarrow\infty$. 
\begin{lemma}
\label{lem:linearized-local-entropy-averages}Let $\tau>0$ and let
$n_{i}=[i^{1+\tau}]$. Then for any $\theta\in\mathcal{P}(G)$ and
$\eta\in\mathcal{P}(\mathbb{R})$, any $\varphi\in\supp\theta$ and
$x\in\supp\eta$, and writing $(A,B)=(A_{\varphi,x},B_{\varphi,x})$
for the derivative of the action map at $(\varphi,x)$, and $y=\varphi(x)$,
we have 
\begin{equation}
\ldim(\eta.\theta,y)\geq\liminf_{k\rightarrow\infty}\frac{1}{k}\sum_{i=0}^{k-1}\frac{1}{n_{i+1}-n_{i}}H(B^{-1}A\theta_{\varphi,n_{i}}*\eta_{x,n_{i}},\mathcal{D}_{n_{i+1}})-\tau.\label{eq:11}
\end{equation}
\end{lemma}
\begin{proof}
This is a combination of (\ref{eq:local-entropy-averages-convolution})
and the linearization argument of Section \ref{sub:Linearization},
which, essentially, allows us to replace the term $H(\theta_{\varphi,n_{i}}\bigdot\eta_{x,n_{i}},\mathcal{D}_{n_{i+1}})$
in (\ref{eq:local-entropy-averages-convolution}) with $H(B^{-1}A\theta_{\varphi,n_{i}}*\eta_{x,n_{i}},\mathcal{D}_{n_{i+1}})$.
In more detail, let $\theta'=\theta_{\varphi,n}$ and $\eta'=\eta_{x,n}$
. The supports of $\theta',\eta'$ are of diameter $O(2^{-n})$, making
the error term in \ref{sub:Linearization} of order $O(2^{-2n})$.
Then, as explained in the paragraph following (\ref{eq:linearization}),
if $m\ll n$ we will have $\frac{1}{m}H(\theta'\bigdot\eta',\mathcal{D}_{n+m})=\frac{1}{m}H(A\theta'*B\eta',\mathcal{D}_{n+m})+O(\frac{1}{m})$.
Taking $n=n_{i}$ and $m=n_{i+1}-n_{i}$, and using $n_{i+1}/n_{i}\rightarrow1$,
we obtain the bound (\ref{eq:11}), where we have moved $B$ from
one side of the convolution to the other by the same argument as before. 
\end{proof}

\subsection{\label{sub:Pushing-entropy-from-G-to-R-and-pointwise-porosity}Pushing
entropy from $G$ to $\mathbb{R}$ and pointwise porosity}

Next, we need a pointwise version of Lemma \ref{lem:entropy-of-image},
which says that large entropy of a component $\theta_{g,i}$ of $\theta\in\mathcal{P}(G)$
translates to large entropy of most push-forwards $\theta_{g,i}\bigdot x$:
\begin{lemma}
\label{lem:pushing-entropy-from-G-to-R-pointwise}Let $(n_{i})$ be
an increasing integer sequence satisfying $\sum_{i=1}^{\infty}(n_{i+1}-n_{i})^{-d}<\infty$
for some $d>0$. Suppose that $\theta\in\mathcal{P}(G)$ and $\eta\in\mathcal{P}(\mathbb{R})$
are compactly supported and further that $\eta(B_{r}(x))\leq ar^{\alpha}$
for some $a,\alpha>0$. Then for $\theta\times\eta$-a.e. $(\varphi,x)$,
we have
\begin{multline*}
\liminf_{k\rightarrow\infty}\frac{1}{k}\sum_{i=0}^{k}\frac{1}{n_{i+1}-n_{i}}H(\theta_{\varphi,n_{i}}\bigdot x,\mathcal{D}_{n_{i+1}})\\
\geq\;\frac{1}{3}\liminf_{k\rightarrow\infty}\frac{1}{k}\sum_{i=0}^{k}\frac{1}{n_{i+1}-n_{i}}H(\theta_{\varphi,n_{i}},\mathcal{D}_{n_{i+1}}^{G}).
\end{multline*}
\end{lemma}
\begin{proof}
Given $\varphi\in\supp\theta$, for each $i$, let 
\[
A_{i}=\left\{ x\in\mathbb{R}\::\,\frac{1}{n_{i+1}-n_{i}}H(\theta_{\varphi,n_{i}}\bigdot x,\mathcal{D}_{n_{i+1}})\leq\frac{1}{3}\cdot\frac{1}{n_{i+1}-n_{i}}H(\theta_{\varphi,n_{i}},\mathcal{D}_{n_{i+1}}^{G})\right\} .
\]
By Lemma (\ref{lem:entropy-of-image}), for every $d>0$ we have $\eta(A_{i})=O(1/(n_{i+1}-n_{i})^{d})$.
Therefore by the assumption on $(n_{i})$, there is a choice of $d$
so that $\sum\eta(A_{i})<\infty$. By Borel-Cantelli, $\eta$-a.e.
$x$ belongs to finitely many $A_{i}$, and for such $x$ the desired
conclusion holds for the given $\varphi$. By Fubini, the conclusion
holds for $\theta\times\eta$-a.e. pair $(\varphi,x)$
\end{proof}
Finally, we need a notion of porosity at a point, in which, instead
of describing the typical behavior of components over the whole measure,
relates only to components containing a fixed point $x$ (i.e. the
components $\eta_{x,i}$) and require that on average they exhibit
porosity. We again do this relative to a subsequence of scales. For
an integer sequence $n_{i}\rightarrow\infty$, we say that $\eta$
is $(h,\delta,m)$-entropy porous along $(n_{i})$ at $x\in\supp\eta$
if 
\begin{equation}
\liminf_{k\rightarrow\infty}\frac{1}{k}\sum_{i=0}^{k}1_{\{\eta_{x,n_{i}}\mbox{ is }(h,\delta,m)\mbox{-entropy porous from scale }n_{i}\mbox{ to }n_{i+1}\}}>1-\delta.\label{eq:21}
\end{equation}
We say that $\eta$ is $h$-entropy porous along $(n_{i})$ at $x$
if for every $\delta>0$ and $m$ it is $(h,\delta,m)$-entropy porous
along $(n_{i})$ at $x$. 
\begin{lemma}
\label{lem:pointwise-porosity-passes-to-ac-measures}Let $(n_{i})$
be a sequence such that $n_{i+1}/n_{i}\rightarrow1$ and $n_{i+1}-n_{i}\rightarrow\infty$,
and suppose that $\eta$ is $(h,\delta,m)$-entropy porous along $(n_{i})$
at $\eta$-a.e. $x$. if $\eta'\ll\eta$ then $\eta'$ is also $(h,\delta,m)$-entropy
porous along $(n_{i})$ at $\eta'$-a.e. $x$. \end{lemma}
\begin{proof}
This follows from the fact that by the martingale convergence theorem,
$\eta'_{x,i},\eta_{x,i}$ are asymptotic in total variation (that
is, $\left\Vert \eta'_{x,i}-\eta_{x,i}\right\Vert \rightarrow0$)
for $\eta'$-a.e. $x$. The details are left to the reader.
\end{proof}

\subsection{\label{sub:hausdorff-dimension-growth}Entropy growth of Hausdorff
dimension under convolution}

We can now state the main result of this section, an analog of Theorem
(\ref{thm:inverse-theorem-for-action}) for Hausdorff dimension.
\begin{theorem}
\label{thm:hausdorff-dimension-growth}For every $\varepsilon>0$
there exists a $\delta'=\delta'(\varepsilon)>0$ such that the following
holds. 

Let $\eta\in\mathcal{P}(\mathbb{R})$ be compactly supported with
$\ldim\eta>0$, and for every $\tau>0$ and $n_{i}=[i^{1+\tau}]$
suppose that $\eta$ is $(1-\varepsilon)$-entropy porous along $(n_{i})$
at $\eta$-a.e. $x$. Then for any $\theta\in\mathcal{P}(G)$, 
\[
\ldim\theta>\varepsilon\quad\implies\quad\ldim\theta\bigdot\eta>\ldim\eta+\delta.
\]
\end{theorem}
\begin{proof}
Fix $\varepsilon>0$, $\theta,\eta$. Let $\delta=\delta(\varepsilon/6)$
be as in Theorem \ref{thm:inverse-theorem-for-action}, and also choose
$m,n$ large enough for that theorem to hold. Write $\alpha=\ldim\eta$
so we are assuming $\alpha>0$. 

Fix $0<\tau<1$ and $n_{i}=[i^{1+\tau}]$. We shall show that $\ldim\theta\bigdot\eta>\ldim\eta+\delta\varepsilon/12-\tau$,
which is enough, since $\tau$ is arbitrary.

First, we claim that we can assume without loss of generality that
there is an $a>0$ and $\beta>0$ such that $\eta(B_{r}(x))\leq ar^{\beta}$
at every $x$. Indeed, given $0<\beta<\alpha$, by Egorov's theorem
we can find disjoint sets $A_{i}$ whose union supports $\eta$, and
such that $\eta(A_{i}\cap B_{r}(x))\leq ar^{\beta}$ for each $i$.
Then $\theta\bigdot\eta=\sum\theta\bigdot(\eta|_{A_{i}})$ and by Lemmas
\ref{lem:hausdorff-dim-stable-under-conditioning} and \ref{lem:pointwise-porosity-passes-to-ac-measures},
it suffices to analyze a single $\eta|_{A_{i}}$, which puts us in
the desired situation.

Let $(\varphi,x)\in G\times\mathbb{R}$ be $\theta\times\eta$-typical
and set $y=\varphi(x)$, which is a $\theta\bigdot\eta$-typical point.
By the local entropy averages lemma (Lemma \ref{lem:local-entropy-averages}),
it suffices for us to show that 
\begin{equation}
\liminf_{k\rightarrow\infty}\frac{1}{k}\sum_{i=1}^{k}\frac{1}{n_{i+1}-n_{i}}H((B^{-1}A\theta_{\varphi,i}\bigdot x)*\eta_{x,n_{i}},\mathcal{D}_{n_{i+1}})\geq\alpha+\frac{\delta\varepsilon}{12}.\label{eq:20}
\end{equation}
For this we shall analyze the behavior of the terms in the average
and show that they are large for a large fraction of $i=1,\ldots,k$,
for all large enough $k$.

For the components $A^{-1}B\theta_{\varphi,i}$, we know that 
\begin{eqnarray*}
\liminf_{k\rightarrow\infty}\frac{1}{k}\sum_{i=1}^{k}\frac{1}{n_{i+1}-n_{i}}H(\theta_{\varphi,n_{i}},\mathcal{D}_{n_{i+1}}^{G}) & = & d(\theta,\varphi)\geq\varepsilon,
\end{eqnarray*}
Because $(\varphi,x)$ is $\theta\times\eta$-typical, by Lemma \ref{lem:pushing-entropy-from-G-to-R-pointwise},
\[
\liminf_{k\rightarrow\infty}\frac{1}{k}\sum_{i=1}^{k}\frac{1}{n_{i+1}-n_{i}}H(\theta_{\varphi,n_{i}}\bigdot x,\mathcal{D}_{n_{i+1}})\geq\frac{\varepsilon}{3},
\]
which, since $B^{-1}A$ is bi-Lipschitz and $n_{i+1}-n_{i}\rightarrow\infty$,
this implies 
\[
\liminf_{k\rightarrow\infty}\frac{1}{k}\sum_{i=1}^{k}\frac{1}{n_{i+1}-n_{i}}H(B^{-1}A\theta_{\varphi,n_{i}}\bigdot x,\mathcal{D}_{n_{i+1}})\geq\frac{\varepsilon}{3}.
\]
Writing
\[
I_{k}=\left\{ 1\leq i\leq k\,:\,H(B^{-1}A\theta_{\varphi,n_{i}}\bigdot x,\mathcal{D}_{n_{i+1}})\geq\frac{\varepsilon}{6}\right\} ,
\]
this give us 
\begin{equation}
\liminf_{k\rightarrow\infty}\frac{1}{k}|I_{k}|\geq\frac{\varepsilon}{6}.\label{eq:25}
\end{equation}

For the components $\eta_{x,i}$, we also know that
\begin{eqnarray*}
\liminf_{k\rightarrow\infty}\frac{1}{k}\sum_{i=1}^{k}\frac{1}{n_{i+1}-n_{i}}H(\eta_{x,n_{i}},\mathcal{D}_{n_{i+1}}) & = & d(\eta,x)\;\geq\;\alpha.
\end{eqnarray*}
Also, fixing a $0<\gamma<\varepsilon/12$ and some $m$, write
\[
J_{k}=\left\{ 1\leq i\leq k\,:\,\eta_{x,n_{i}}\mbox{ is not }(1-\varepsilon,\gamma,m)\mbox{-entrpoy porous from scale }n_{i}\mbox{ to }n_{i+1}\right\} .
\]
Then by assumption 
\begin{equation}
\limsup_{k\rightarrow\infty}\frac{1}{k}|J_{k}|<\gamma.\label{eq:26}
\end{equation}
For $i\in I_{k}\setminus J_{k}$, we can apply Theorem \ref{thm:inverse-theorem-for-action}
and conclude that 
\[
\frac{1}{n_{i+1}-n_{i}}H((B^{-1}A\theta_{\varphi,i}\bigdot x)*\eta_{x,n_{i}},\mathcal{D}_{n_{i+1}})\geq\frac{1}{n_{i+1}-n_{i}}H(\eta_{x,n_{i}},\mathcal{D}_{n_{i+1}})+\delta.
\]
Finally, by (\ref{eq:25}) and (\ref{eq:26}), for $k$ large enough,
$\frac{1}{k}|I_{k}\setminus J_{k}|\geq\varepsilon/12$, and so by
the last inequality we can estimate (\ref{eq:20}) by
\begin{multline*}
\liminf_{k\rightarrow\infty}\frac{1}{k}\sum_{i=1}^{k}\frac{1}{n_{i+1}-n_{i}}H((B^{-1}A\theta_{\varphi,i}\bigdot x)*\eta_{x,n_{i}},\mathcal{D}_{n_{i+1}})\\
\begin{aligned}\geq & \liminf_{k\rightarrow\infty}\left(\frac{1}{k}\sum_{i\in I_{k}}^{k}\frac{1}{n_{i+1}-n_{i}}H(\eta_{x,n_{i}},\mathcal{D}_{n_{i+1}})+\delta\cdot\frac{1}{k}|I_{k}\setminus J_{k}|\right)\\
\geq & \ldim(\eta,x)+\delta\cdot\frac{\varepsilon}{12}
\end{aligned}
\end{multline*}
and we are done. 
\end{proof}

\subsection{\label{sub:Proof-of-porousity-theorem}Proof of Theorem \ref{thm:convolution-porous-sets}}

Let $X\subseteq\mathbb{R}$ be a compact $c$-porous set and $\Phi\subseteq G$
compact with $\dim\Phi>c$. We now prove that $\dim\Phi\bigdot X>\dim X+\delta$
for some $\delta=\delta(c)>0$. 

First, choose $\nu\in\mathcal{P}(\Phi)$ with $\ldim\nu>c$, which,
since $\dim\Phi>c$, exists by Frostman's lemma. 

Second, note that by porosity of $X$, any $\mu\in\mathcal{P}(X)$
is $(1-c')$-entropy porous for some $c'$ depending only on $c$,
and furthermore $\mu$ is $(1-c')$-entropy porous at every $x\in\supp\mu$
(along any sequence of scales).\footnote{To see this note that for $m$ such that $2^{-m}<c/2$, any dyadic
interval of length $2^{-i}$ contains a dyadic interval of length
$2^{-(i+m)}$ disjoint from $X$. Therefore, for any component any
$\mu_{x,i}$ of $\mu\in\mathcal{P}(X)$, we have $H(\mu_{x,i},\mathcal{D}_{i+m})\leq\log(2^{m}-1)/m<1$.
The porosity statements follow from this. } Let $\delta>0$ be the parameter $\delta'$ supplied by in Theorem
\ref{thm:hausdorff-dimension-growth} for $\varepsilon=\min\{c',c\}$
(so $\delta$ depends only on $c$), and use Frostman's lemma again
to find $\mu\in\mathcal{P}(X)$ with $\ldim\mu>\dim X-\delta/2$ . 

Now, by Theorem \ref{thm:hausdorff-dimension-growth},
\[
\ldim\nu\bigdot\mu>\ldim\mu+\delta>\dim X+\frac{\delta}{2}.
\]
Since, $\nu\bigdot\mu$ is supported on $\Phi\bigdot X$, so we get $\dim\Phi\bigdot X>\dim X+\delta/2$,
as desired.

\section{\label{sec:Conjecture-implies-conjecture}Conjecture \ref{conj:overlaps}
implies Conjecture \ref{conj:main}}

\subsection{A Tits-like alternative for semigroups}

In this section we prove Theorem \ref{thm:implication-btwn-conjectures},
which asserts that Conjecture \ref{conj:overlaps} implies Conjecture \ref{conj:main}.
The main idea is use largeness of $\Phi$ to show that $\Phi$, or
some power of it, contains an infinite free set (i.e. a set freely
generating a semigroup). The largeness we require is expressed both
in terms of the cardinality of $\Phi$ and its algebraic properties;
specifically, we require that it not be contained in too small a subgroup
of $G$. Recall that a group is said to be virtually abelian if it
contains a finite-index abelian subgroup. It is not too hard to show
that every virtually abelian subgroups of $G$ is contained either
in the isometry group, or in the stabilizer group of some point (this
can be derived from Lemma \ref{lem:1-param-subgoup-containing-g}
below). With these assumptions we will prove:
\begin{proposition}
\label{prop:large-free-sets}Suppose that $\Phi\subseteq G$ is uncountable
and is not contained in a virtually abelian subgroup. Then there exists
a $k\in\mathbb{N}$ such that $\Phi^{k}=\{\varphi_{1}\circ\ldots\circ\varphi_{k}\,:\,\varphi_{i}\in\Phi\}$ contains an infinite
free set.
\end{proposition}
The fact that all generators lie in the same power $\Phi^{k}$ is
important (it is much simpler to show that $\bigcup_{k=1}^\infty\Phi^k$ contains an infinite free set). Related (and much deeper) statements exist in the context
of the classical Tits alternative, see e.g. \cite{BreuillardGelander2008},
but they do not seem to give what we need here.

Assuming this proposition, we can prove the implication between the
conjectures:
\begin{proof}[Proof of Thoerem \ref{thm:implication-btwn-conjectures}] Fix a compact
uncountable $\Phi\subseteq\mathcal{S}$ whose attractor $X$ is not
a single point. Using compactness of $\Phi$ we can find $0<r_{0}<1$
such that $\left\Vert \varphi\right\Vert \geq r_{0}$ for all $\varphi\in\Phi$.
Now, $\Phi$ is not contained in the $G$-stabilizer of a point $x_{0}$
(since otherwise we would have $X=\{x_{0}\}$ contrary to assumption),
nor in the isometry group (since $\Phi$ consists of contractions),
so $\Phi$ is not contained in a virtually abelian subgroup. By Proposition
\ref{prop:large-free-sets} there exists a $k$ such that $\Phi^k$
contains an infinite free set. In particular for $\ell=\left\lceil 1/r_{0}^{2k}\right\rceil $
there is a free subset $\Phi_{0}\subseteq\Phi^{ k}$ of size $\ell$.
Since $\left\Vert \varphi_{i}\right\Vert \geq r_{0}^{2k}$ for all
$\varphi\in\Phi^{ k}$, for any $s\leq1$ we have 
\[
\sum_{\varphi\in\Phi_{0}}\left\Vert \varphi\right\Vert ^{s}\geq\sum_{\varphi\in\Phi_{0}}r_{0}^{2ks}\geq\ell r_{0}^{2ks}>\ell r_{0}^{2k}\geq1,
\]
showing that $s(\Phi_{0})\geq1$. By Conjecture \ref{conj:overlaps},
the attractor $X_{0}$ of $\Phi_{0}$ satisfies $\dim X_{0}=\min\{1,s(\Phi_{0})\}=1$.
Since $X_{0}\subseteq X$ we have $\dim X=1$, giving conjecture \ref{conj:main}.
\end{proof}
We present the proof of the proposition, which is elementary but not
short, over the next few sections.

Throughout, we parametrize $G$ as a subset of $\mathbb{R}^{2}$,
identifying $\varphi(x)=sx+t$ with $(s,t)\in\mathbb{R}^{2}$. This
parametrization differs from that used in previous sections but it
simplifies some of the algebraic considerations.

\subsection{Subgroups of $G$}

For most of the proof we work in the group $G^{+}$ of orientation-preserving
similarities of $\mathbb{R}$. In parameter space, this is the subset
$(0,\infty)\times\mathbb{R}$. 

A one-parameter subgroup of $G^{+}$ is the image of a continuous
injective homomorphism $\mathbb{R}\rightarrow G$. There are two types
of examples: First, the group of translations $x\mapsto x+t$ for
$t\in\mathbb{R}$; and second, for each $x_{0}$, the $G^{+}$-stabilizer
of $x_{0}$, consisting of maps $x\mapsto s(x-x_{0})+x_{0}$, $s>0$.

Observe that a similarity has no fixed point if and only if it is
a non-trivial translation, and if it is not a translation, then the
fixed point is unique (this is just because the equation $sx+t=x$
has no solution if $s=1$ and $t\neq0$, and precisely one solution
if $s\neq1$). Thus every non-trivial element of $G^{+}$ belongs
either to the translation group, or to a stabilizer group, but not
both. Also, by uniqueness of the fixed point, the stabilizer groups
of different points can intersect only in the identity. This shows
that the translation and stabilizer groups cover all of $G^{+}$ but
any two meet only at the identity.
\begin{lemma}
\label{lem:normality}\label{lem:1-param-subgoup-containing-g}If
$H\leq G^{+}$ is a 1-parameter subgroup then it is either the translation
group or a stabilizer group, and in the latter case, $\varphi H\varphi^{-1}\cap H=\{\id\}$
for all $\varphi\in G^{+}\setminus H$.\end{lemma}
\begin{proof}
Let $H\leq G^{+}$ be a 1-parameter subgroup not contained in the
translation group. Then there is some $\psi\in H$ with a fixed point
$y_{0}$. If $\varphi\in H$ then $\varphi\psi\varphi^{-1}$ fixes
$\varphi(y_{0})$, but since $H$ is abelian, $\varphi\psi\varphi^{-1}=\psi$,
so it fixes $y_{0}$. By uniqueness of the fixed point we have $\varphi(y_{0})=y_{0}$,
so $\varphi$ belongs to the stabilizer group $H'$ of $y_{0}$. Since
$\varphi\in H$ was arbitrary this shows that $H\leq H'$, and since
$H'$ is isomorphic to $\mathbb{R}$ it has no nontrivial closed subgroups,
so $H=H'$. This proves the first statement.

For the second statement, let $H$ be the stabilizer of $y_{0}$,
and $\varphi\in G^{+}\setminus H$, so by definition $\varphi(y_{0})\neq y_{0}$.
Given any $\id\neq\psi\in H$, the unique fixed point of $\varphi\psi\varphi^{-1}$
is $\varphi(y_{0})\neq y_{0}$, which shows $\varphi\psi\varphi^{-1}\notin H$.
Since $\psi\in H$ was arbitrary, this implies $\varphi H\varphi^{-1}\cap H=\{\id\}$.\end{proof}
\begin{lemma}
\label{lem:parametrization-of-subgroups}Every 1-parameter subgroup
of $G^{+}$ is given in parameter space by the intersection of a line
with $(0,\infty)\times\mathbb{R}$.\end{lemma}
\begin{proof}
Writing $\varphi(x)=sx+t$ for a general element of $G^{+}$, the
translation group is given by the equation $s=1$, and the stabilizer
of $x_{0}$ by the equation $sx_{0}+t=x_{0}$ (and $s>0$). These
are the only 1-parameter groups by the previous lemma.
\end{proof}

\subsection{A class of curves and their stabilizers}

Let $\mathcal{C}$ denote the collection of subsets $\Gamma\subseteq G^{+}$
which are either singletons, lines (i.e. in coordinates they are determined
by a linear equation), or in coordinates have the form $\{(s,p(s)/q(s))\,:\,s>0\,,\,q(s)\neq0\}$
for some real polynomials $p,q$. An easy computation shows that $\mathcal{C}$
is closed under the action of $G^{+}$ by pre- and post-composition.
It is also easy to check that if $\Gamma_{1},\Gamma_{2}\in\mathcal{C}$,
then either $\Gamma_{1}\cap\Gamma_{2}=\Gamma_{1}=\Gamma_{2}$ or else
$\Gamma_{1}\cap\Gamma_{2}$ is finite.

By Lemma \ref{lem:parametrization-of-subgroups}, every 1-parameter
subgroup of $G^{+}$ is in $\mathcal{C}$. Given $\Gamma\in\mathcal{C}$,
set 
\[
G_{\Gamma}=\{g\in G^{+}\,:\,\Gamma g\subseteq\Gamma\}.
\]

\begin{lemma}
\label{lem:Gamma-stabilizer-is-in-C}If $\Gamma\in\mathcal{C}$ then
either $G_{\Gamma}=\{\id\}$ or $G_{\Gamma}\in\mathcal{C}$ is a 1-parameter
group and $\Gamma=\gamma G_{\Gamma}$ is a coset.\end{lemma}
\begin{proof}
Suppose $\id\neq g\in G_{\Gamma}$ and let $H\leq G^{+}$ be the 1-parameter
subgroup containing $g$, so $H\in\mathcal{C}$. Fix $\gamma\in\Gamma$,
so that $\gamma g^{n}\in\Gamma$ for all $n\in\mathbb{N}$ and all
these elements are distinct. Hence $\{\gamma g^{n}\}\subseteq\gamma H\cap\Gamma$,
so $\gamma H\cap\Gamma$ is infinite. Since $\gamma H,\Gamma\in\mathcal{C}$,
we conclude that $\gamma H\cap\Gamma=\Gamma=\gamma H$.

Finally, if $\id\neq g'\in G_{\Gamma}$ and $H'$ is the 1-parameter
group containing $g'$, then by the same argument, $\Gamma=\gamma H'$.
Thus $\gamma H=\gamma H'$, so $H=H'$ and in particular $g'\in H$.
Since $g'\in G_{\Gamma}$ was arbitrary we conclude that $G_{\Gamma}=H$,
and $\Gamma=\gamma G_{\Gamma}$. \end{proof}
\begin{corollary}
If $\Gamma\in\mathcal{C}$ then $G_{\Gamma}=\{g\in G^{+}\,:\,\Gamma g=\Gamma\}$.
\end{corollary}

\subsection{\label{sub:Relations}Relations}

A word $w(z_{1},\ldots,z_{n})$ over the letters $z_{1},\ldots,z_{n}$
is a finite formal product of the letters, $z_{i_{1}}z_{i_{2}}\ldots z_{i_{N}}$
in which all variables appear. For a sequence of elements $\varphi_{1},\ldots,\varphi_{n}\in G^{+}$
we write $w(\varphi_{1},\ldots,\varphi_{n})=\varphi_{i_{1}}\varphi_{i_{2}}\ldots\varphi_{i_{n}}$
for the group element obtained by substituting $\varphi_{i}$ for
$z_{i}$ in the formal product. We say that $\Phi_{0}\subseteq G^{+}$
is free if $w(\varphi_{1},\ldots,\varphi_{m})=w'(\psi_{1},\ldots,\psi_{n})$
and $\varphi_{i},\psi_{i}\in\Phi_{0}$ implies $w=w'$ (this implies
that the semigroup generated by $\Phi_{0}$ is free, not necessarily
the group; for groups, we would need to allow inverses and consider
reduced words).

Given words $w,w'$ and $\varphi_{i},\psi_{i}\in\Phi_{0}$, we are
interested in describing the set of $\gamma\in G^{+}$ which satisfy
the relation $w(\varphi_{0},\ldots,\varphi_{m},\gamma)=w'(\psi_{0},\ldots,\psi_{n},\gamma)$.
If such an equality holds for some $w\neq w'$ we say that $\gamma$
satisfies a relation over $\Phi_{0}$. We begin by considering words
in a certain canonical form.
\begin{proposition}
\label{prop:solution-of-alternating-relation}Let $\{\varphi_{2i}\}_{i=0}^{m}$
and $\{\psi_{2j}\}_{j=0}^{n}$ be sequences of elements of $G^{+}$
and let $\Gamma$ denote the set of all $\gamma\in G^{+}$ satisfying
\begin{equation}
\varphi_{0}\gamma\varphi_{2}\gamma\varphi_{4}\ldots\varphi_{2m-2}\gamma\varphi_{2m}=\psi_{0}\gamma\psi_{2}\gamma\psi_{4}\ldots\psi_{2n-2}\gamma\psi_{2n}.\label{eq:relation}
\end{equation}
Then either $\Gamma$ is empty, or it is a finite union of elements
of $\mathcal{C}$, or $\Gamma=G^{+}$; and the latter occurs if and
only if $m=n$ and $\varphi_{i}=\psi_{i}$ for all $i=0,\ldots,n$.\end{proposition}
\begin{proof}
Suppose that $\gamma(x)=sx+t$ is a solution and set $\varphi_{2i+1}=\psi_{2i+1}=\gamma$,
so that the assumption is that $\varphi_{0}\varphi_{1}\ldots\varphi_{2m}=\psi_{0}\psi_{1}\ldots\cdot\psi_{2n}$.
Write $\varphi_{i}(x)=a_{i}x+b_{i}$ and $\psi_{i}(x)=c_{i}x+d_{i}$,
in particular $a_{2i+1}=c_{2i+1}=s$ and $b_{2i+1}=d_{2i+1}=t$. We
compute the product explicitly:
\begin{eqnarray}
\varphi_{0}\varphi_{1}\ldots\varphi_{2m}(x) & = & b_{2m}+a_{2m}(b_{2m-1}+a_{2m-1}(b_{2m-1}+a_{2m-2}(\ldots(x))))\nonumber \\
 & = & (\prod_{i=0}^{2m}a_{i})x+\sum_{i=0}^{2m}b_{i}\cdot(\prod_{j=i+1}^{2m}a_{i})\nonumber \\
 & = & s^{m}(\prod_{i=0}^{m}a_{2i})x+\label{eq:composition}\\
 &  & \quad+\left(t\left(\sum_{i=1}^{m}s^{m-i}\prod_{\ell=i}^{m}a_{2\ell}\right)+\right.\nonumber \\
 &  & \quad\quad\quad+\left.\sum_{i=0}^{m}\left(b_{2i}\cdot s^{m-i}\cdot\prod_{\ell=i+1}^{m}a_{2\ell}\right)\right),\nonumber 
\end{eqnarray}
where in the last equality we simply separated out the term containing
$x$, the terms containing $t$, and the rest. The corresponding formula
holds for $\psi_{0}\ldots\psi_{2n}$. 

Thus, in order for $\varphi_{0}\ldots\varphi_{2m}=\psi_{0}\ldots\psi_{2n}$,
we must have agreement between the coefficient of $x$ and the constant
term in each product. The first of these conditions translates to
\begin{equation}
s^{m}(\prod_{i=0}^{m}a_{2i})=s^{n}(\prod_{i=0}^{n}c_{2i}).\label{eq:x-coefficient}
\end{equation}
There is either a unique positive solution $s$, or, if $m=n$ and
$\prod_{i=0}^{n}a_{2i}=\prod_{i=0}^{n}c_{2i}$, every $s$ is a solution. 

The equality of the constant terms (those not involving $x$) yields
an equation $e(s)t+f(s)=0$ in which the coefficients $e(s),f(s)$
are given by 
\begin{eqnarray*}
e(s) & = & \sum_{i=1}^{m}\left(s^{m-i}\prod_{\ell=i}^{m}a_{2\ell}\right)-\sum_{i=1}^{n}\left(s^{n-i}\prod_{\ell=i}^{n}c_{2\ell}\right)\\
f(s) & = & \sum_{i=0}^{m}\left(b_{2i}\cdot s^{m-i}\cdot\prod_{\ell=i+1}^{m}a_{2\ell}\right)-\sum_{i=0}^{n}\left(d_{2i}\cdot s^{n-i}\cdot\prod_{\ell=i+1}^{n}c_{2\ell}\right),
\end{eqnarray*}
which are polynomial in $s$. 

If every $s$ solves (\ref{eq:x-coefficient}), we distinguish two
cases.
\begin{description}
\item [{Case~1}] $e(s)$ is not identically zero. Then for every $s$
such that $e(s)\neq0$, the equation $e(s)t+f(s)=0$ has the unique
solution $t=-f(s)/e(s)$, and we have found the curve $(s,-f(s)/e(s))$
in solution space. There may also be finitely many values of $s$
for which $e(s)=0$. For such $s$, if $f(s)\neq0$ there is no solution,
while if $f(s)=0$ any $t$ is a solution, and we have found a line
in solution space. 
\item [{Case~2}] $e(s)$ is identically zero. Since $a_{i},c_{i}\neq0$
this can happen only if $m=n$. Then by comparing coefficients we
find by induction $a_{2i}=c_{2i}$ for all $i=1,\ldots,n$. Since
we are assuming that (\ref{eq:x-coefficient}) is a trivial equation,
$\prod_{i=0}^{n}a_{2i}=\prod_{i=0}^{n}c_{2i}$, and the corresponding
terms are equal and non-zero for $i\geq1$, they are equal also for
$i=0$, and we find that $a_{2i}=c_{2i}$ for all $i=0,\ldots,n$.
Next, if $b_{2i}=d_{2i}$ for all $i$, then we would have $\varphi_{2i}=\psi_{2i}$
for all $i$, and the solution space is all of $G^{+}$. Otherwise
there is an $i$ with $b_{2i}\neq d_{2i}$, and this, together with
$a_{2i}=c_{2i}$ for all $i$, implies that $f(s)$ is not the zero
polynomial. Recalling that $e(s)=0$ for all $s$, our equation has
become $0t+f(s)=0$, which can be solved only when $f(s)=0$. This
occurs for finitely many values of $s$, and when it does, any $t$
solves the equation, giving a line in solution space. 
\end{description}
On the other hand, suppose (\ref{eq:x-coefficient}) has a unique
solution $s_{0}$. Then the solution set $\Gamma$ of the original
relation consists of those $(s_{0},t)\in(0,\infty)\times\mathbb{R}$
for which $t$ satisfies $e(s_{0})t+f(s_{0})=0$. This equation either
has no solutions, one solution $t_{0}$ (in which case $\Gamma=\{(s_{0},t_{0})\}$,
or else every $t$ solves it, in which case $\Gamma$ is the line
$s=s_{0}$.

Examining the result in each of the cases, we find we have proved
the proposition.\end{proof}
\begin{corollary}
\label{cor:reltaions-with-free-parameters}Let $w(z_{0},\ldots,z_{m},z)$
and $w'(z_{0},\ldots,z_{n},z)$ be words, let $\Phi_{0}\subseteq G^{+}$
be a free set, and let $\varphi_{0},\ldots,\varphi_{m},\psi_{0},\ldots,\psi_{n}\in\Phi_{0}$.
Let $\Gamma$ be the set of $\gamma\in G^{+}$ such that $w(\varphi_{0}\ldots\varphi_{m},\gamma)=w(\psi_{0},\ldots,\psi_{n},\gamma)$.
Then either $\Gamma=G^{+}$, in which case $m=n$ and $w=w'$, or
$\Gamma$ is a finite union of elements of $\mathcal{C}$.\end{corollary}
\begin{proof}
By multiplying together consecutive occurrences of the $\varphi_{i},\psi_{i}$,
breaking occurrences of $\gamma^{k}$ into $\gamma\id\gamma\id\ldots\id\gamma$,
and inserting if necessary the identity at the beginning and end of
the product, $\Gamma$ becomes the set of $\gamma$ satisfying a relation
\[
\varphi'_{0}\gamma\varphi'_{2}\gamma\ldots\varphi'_{2(m'-1)}\gamma\varphi'_{2m'}=\psi'_{0}\gamma\psi'_{2}\gamma\ldots\psi'_{2(n'-1)}\gamma\psi'_{2n'}
\]
with each $\varphi'_{i},\psi'_{i}$ either the identity or a product
of the original $\varphi_{i},\psi_{i}$. By the proposition, either
$\Gamma$ is a finite union of elements of $\mathcal{C}$ or $\Gamma=G^{+}$,
in which case $m'=n'$ and $\varphi'_{i}=\psi'_{i}$. In the latter
case, because $\Phi_{0}$ is free, this means that each $\varphi'_{i}=\psi'_{i}$
decomposes uniquely as a product of the original $\varphi_{i},\psi_{i}$,
and we conclude that $m=n$ and $\varphi_{i}=\psi_{i}$ as claimed.
\end{proof}

\subsection{Proof of Proposition \ref{prop:large-free-sets}: Cosets of the translation
group}

We first prove a special case of Proposition \ref{prop:large-free-sets}
in which $\Phi\subseteq G$ is contained in a coset of the translation
group, or equivalently, there is some common $a\neq0$ such that all
$\varphi\in\Phi$ are of the form $x\mapsto ax+b$ for some $b$. 

If $\varphi_{1},\ldots,\varphi_{m},\psi_{1},\ldots,\psi_{n}\in\Phi$
satisfy $\varphi_{1}\ldots\varphi_{m}=\psi_{1}\ldots\psi_{n}$, then,
writing $\varphi_{i}(x)=ax+b_{i}$ and $\psi_{j}(x)=ax+d_{j}$, by
a similar calculation to the one in (\ref{eq:composition}) we have
\begin{equation}
a^{m}+\sum_{i=1}^{m}a^{m-i}b_{i}=a^{n}+\sum_{i=1}^{n}a^{n-i}d_{i}.\label{eq:relation-constant-contraction}
\end{equation}
Let $E=1\cup\{b_{i}\}\cup\{d_{i}\}$ and note that the union may not
be disjoint. Grouping together the coefficients of each $e\in E$
in the last equation, we obtain an equation of the form
\[
\sum_{e\in E}p_{e}(a)\cdot e=0,
\]
where $p_{e}(\cdot)$ is a polynomial with coefficients $\pm1$ and
$0$. If $a$ is not the root of any polynomial of this kind, this
implies each $b_{i}$ is in the field generated by the other $b_{j},d_{j}$
and $a$. Thus we can produce an infinite free set in $\Phi$ simply
by choosing the $i$-th map $x\mapsto ax+b_{i}$ in such a way that
$b_{i}\in\mathbb{R}\setminus\mathbb{Q}(a,b_{1},\ldots,b_{i-1})$,
which we can do because $\Phi$ is uncountable.

However, when $a$ is the root of a polynomial with coefficients $\pm1,0$,
this argument and its conclusion fail. For example, consider either
of the roots $a=(1\pm\sqrt{5})/2$ of the equation $x^{2}-x-1=0$,
and the words 
\begin{eqnarray*}
w(z_{1},z_{2}) & = & z_{1}z_{2}z_{2}\\
w'(z_{1},z_{2}) & = & z_{2}z_{1}z_{1}.
\end{eqnarray*}
Then for any $\varphi(x)=ax+b$ and $\psi(x)=ax+d$, the relation
$w(\varphi,\psi)=w'(\psi,\varphi)$ is equivalent to 
\[
b+ad+a^{2}d+a^{3}=d+ab+a^{2}b+a^{3}.
\]
Rearranging we get 
\[
b(a^{2}-a-1)=d(a^{2}-a-1).
\]
Since $a^{2}-a-1=0$, every $b,d$ satisfy this, so $\{\varphi,\psi\}$
is not free for any choice of $b,d$. 

We can avoid this problem by taking finitely many powers. 
\begin{proposition}
\label{prop:free-set-in-monocontracting-family}If $\Phi\subseteq G$
is uncountable and contained in a non-trivial coset of the translation
group, then there exists a $k\in\mathbb{N}$ such that $\Phi^{k}$
contains an infinite free subset.\end{proposition}
\begin{proof}
Suppose $a\in\mathbb{R}$ and all elements of $\Phi$ are of the form
$x\mapsto ax+b$ for some $b$. All non-zero roots of a polynomial
with coefficients $\pm1,0$ have modulus in the range $(\frac{1}{2},2)$,
so if $|a|\geq2$ or $|a|\leq\frac{1}{2}$ we can use the construction
discussed above to obtain an infinite free subset of $\Phi$. Otherwise
set $k=\left\lceil |\log_{2}|a||\right\rceil $ and note that every
$\varphi\in\Phi^{k}$ is of the form $x\mapsto a^{k}x+b$ for some
$b$, and $|a^{k}|\notin(\frac{1}{2},2)$ by choice of $k$, so by
the same argument we can find a free subset of $\Phi^{k}$.
\end{proof}

\subsection{Proof of Proposition \ref{prop:large-free-sets}: Other cosets}
\begin{proposition}
\label{prop:free-set-in-coset}If $\Phi\subseteq G^{+}$ is uncountable
and contained in a non-trivial coset $\varphi_{0}H$ of a 1-parameter
subgroup $H$ other than the translation group. Then $\Phi^{2}$ contains
an infinite free subset.\end{proposition}
\begin{proof}
Write $F=\varphi_{0}H$; we first claim that the collection $\{\varphi F\}_{\varphi\in\Phi}$
is pairwise disjoint. If not, then $\varphi_{1}f_{1}=\varphi_{2}f_{2}\neq\emptyset$
for some distinct pair $\varphi_{1},\varphi_{2}\in\Phi$ and some
$f_{1},f_{2}\in F$. Since $\Phi\subseteq\varphi_{0}H=F$ we can write
$\varphi_{i}=\varphi_{0}h_{i}$ for some distinct $h_{1},h_{2}\in H$,
and $f_{i}=\varphi_{0}\overline{h}_{i}$ for some $\overline{h}_{1},\overline{h}_{2}\in H$.
Thus 
\begin{equation}
\varphi_{0}h_{1}\varphi_{0}\overline{h}_{1}=\varphi_{0}h_{2}\varphi_{0}\overline{h}_{2}.\label{eq:1}
\end{equation}
Since $h_{1}\neq h_{2}$ we conclude that $\overline{h}_{1}\neq\overline{h}_{2}$.
But rearranging (\ref{eq:1}) gives $\varphi_{0}^{-1}h_{2}^{-1}h_{1}\varphi_{0}=\overline{h}_{2}\overline{h}_{1}^{-1}\neq\id$,
showing that $\varphi_{0}^{-1}H\varphi_{0}\cap H\neq\{\id\}$. By
Lemma \ref{lem:normality} this can occur only if $H$ is the translation
group, contrary to our assumption.

It suffices for us to show that given a finite free subset $\Delta\subseteq G^{+}$
we can find $\gamma\in\Phi^{2}$ such that $\Delta\cup\{\gamma\}$
is free, since we can then build an infinite free set by induction.
Fix $\Delta$. By Corollary \ref{cor:reltaions-with-free-parameters},
the set of all $\gamma\in G^{+}$ such that $\Delta\cup\{\gamma\}$
is not free is a countable union of sets $\Gamma_{1},\Gamma_{2},\ldots\in\mathcal{C}$,
so $\Delta\cup\{\gamma\}$ is free for any $\gamma\in\Phi^{2}\setminus\bigcup_{i=1}^{\infty}\Gamma_{i}$.
Thus, our goal is to show that $\Phi^{2}\setminus\bigcup_{i=1}^{\infty}\Gamma_{i}\neq\emptyset$.
Now, if $\varphi F\cap\Gamma_{i}$ is infinite for some $\varphi$
and $i$, then $\varphi F=\Gamma_{i}$ (because both sets are in $\mathcal{C}$),
hence, since $\{\varphi F\}_{\varphi\in\Phi}$ is pairwise disjoint,
for each $i$ there is at most one $\varphi\in\Phi$ such that $\varphi F\cap\Gamma_{i}$
is infinite. Therefore, since $\Phi$ is uncountable, there must be
some $\varphi\in\Phi$ such that $\varphi F\cap\Gamma_{i}$ is finite
for all $i$. Since $\Phi\subseteq F$ also $\varphi\Phi\cap\Gamma_{i}$
is finite, and since $\Phi$ is uncountable, $\varphi\Phi\setminus\bigcup_{i=1}^{\infty}\Gamma_{i}\neq\emptyset$,
as desired.
\end{proof}

\subsection{Proof of Proposition \ref{prop:large-free-sets}: Orientation-preserving
case}
\begin{proposition}
If $\Phi\subseteq G^{+}$ is uncountable and is not contained in a
1-parameter subgroup then either $\Phi$ or $\Phi^2$ contains an infinite free subset.\end{proposition}
\begin{proof}
  It suffices to show that $\Phi\cup\Phi^2$ contains an infinite free subset.  
To do this it suffices to show that, given a finite free set
$\Delta\subseteq G^{+}$, there is a $\gamma\in\Phi\cup\Phi^2$ such
that $\Delta\cup\{\gamma\}$ is free. Fix $\Delta$, and define $\Gamma_{0},\Gamma_{1},\Gamma_{2},\ldots\in\mathcal{C}$
as in the proof of the previous proposition, so we must show that
there is a $k$ with $\Phi\cup\Phi^2\not\subseteq\bigcup_{i=0}^{\infty}\Gamma_{i}$. 

If $\Phi\cap\Gamma_{i}$ is countable for all $i$ we are done, since
$\Phi$ is uncountable. So suppose one of the intersections is uncountable;
without loss of generality it is $\Phi\cap\Gamma_{0}$, and write
$\Phi_{0}=\Phi\cap\Gamma_{0}$.

If for some $\varphi\in\Phi_{0}$ we have $\Phi_{0}\varphi\not\subseteq\bigcup\Gamma_{i}$
then we are done, so assume the contrary. Then for each $\varphi\in\Phi$
we have $\Phi\varphi\subseteq\bigcup_{i=0}^{\infty}\Gamma_{i}$ and
by another cardinality argument there is some $i=i(\varphi)$ such
that $\Phi_{0}\varphi\cap\Gamma_{i}$ is uncountable, and since $\Phi_{0}\varphi\cap\Gamma_{i}\subseteq\Gamma_{0}\varphi\cap\Gamma_{i}$
also $\Gamma_{0}\varphi\cap\Gamma_{i}$ is uncountable. Since both
sets are in $\mathcal{C}$ we conclude that $\Gamma_{0}\varphi=\Gamma_{i(\varphi)}$. 

Since $\Phi_{0}$ is uncountable, there must be distinct $\varphi,\psi\in\Phi_{0}$
such that $i(\varphi)=i(\psi)$, i.e. $\Gamma_{0}\varphi=\Gamma_{0}\psi$,
or equivalently, $\Gamma_{0}=\Gamma_{0}\psi\varphi^{-1}$. Thus $\psi\varphi^{-1}$
is a non-trivial element of $G_{\Gamma_{0}}$, so by Lemma \ref{lem:Gamma-stabilizer-is-in-C},
$\Gamma_{0}$ is a coset $\varphi_{0}H$ of the 1-parameter group
$H=G_{\Gamma_{0}}$. 

If $\varphi_{0}\notin H$ (i.e. $\Gamma_{0}\neq H$) we are done,
since we have $\Phi_{0}\subseteq\varphi_{0}H$ and we can apply Proposition
\ref{prop:free-set-in-monocontracting-family} or \ref{prop:free-set-in-coset}
to $\Phi_{0}$. 

Otherwise $\varphi_{0}\in H$ and $\Phi_{0}\subseteq H$. But by hypothesis
$\Phi\not\subseteq H$, and we can choose $\psi\in\Phi\setminus H$.
Then $\psi\Phi_{0}\subseteq\psi H$, and we can again apply Proposition
\ref{prop:free-set-in-monocontracting-family} or \ref{prop:free-set-in-coset}
to $\psi\Phi_{0}$. Since $\psi\Phi_{0}\subseteq\Phi^{2}$ this gives
the claim.
\end{proof}

\subsection{Proof of Proposition \ref{prop:large-free-sets}: general case}

Suppose that $\Phi\subseteq G$ is uncountable and not virtually abelian.
We shall show that $\Phi^{2}\cap G^{+}$ is uncountable and not contained
in a 1-parameter group. This is enough, since we can then apply the
results of the previous section to $\Phi^{2}\cap G^{+}$.

Observe that the group $H$ of all $\varphi\in G$ fixing a given
$x_{0}\in\mathbb{R}$ is virtually abelian, since $H\cap G^{+}$ is
abelian and $H\cap G^{+}$ is the kernel of the homomorphism $H\rightarrow\{\pm1\}$
mapping $\varphi(x)=$$ax+b$ to $\sgn a$, implying that $H\cap G^{+}$
has index two in $H$. Similarly, the isometry group of $\mathbb{R}$
contains the group of translations as an index two subgroup, so it
is also virtually abelian. In particular, we conclude that $\Phi$
is neither contained in the isometry group, nor in the stabilizer
in $G$ of any $x_{0}\in\mathbb{R}$.

Since either $\Phi\cap G^{+}$ or $\Phi\cap(G\setminus G^{+})$ must
be uncountable, and the square of each of these sets is contained
in both $\Phi^{2}$ and $G^{+}$, certainly $\Phi^{2}\cap G^{+}$
is uncountable. Let $\id\neq\varphi\in\Phi^{2}\cap G^{+}$.

Suppose that $\varphi$ is a translation. Since $\Phi$ is not contained
in the isometry group there is a $\psi\in\Phi$ which is not an isometry.
But then $\psi^{2}\in\Phi^{2}$ also is not an isometry, and $\psi^{2}\in G^{+}$.
Thus $\Phi^{2}\cap G^{+}$ contains both translations and nontrivial
elements that are not translations, so $\Phi^{2}\cap G^{+}$ is not
contained in a 1-parameter group.

Otherwise, since $\varphi$ is not a translation or the identity,
it fixes some point $x_{0}$. Since $\Phi$ is not contained in the
stabilizer group of $x_{0}$, there is some $\psi\in\Phi$ that does
not fix $x_{0}$. Then $\psi^{2}$ also does not fix $x_{0}$ (for
either $\psi$ fixes another point, and $\psi^{2}$ does as well,
or else $\psi$ was already a translation without fixed points, and
then $\psi^{2}$ is too). Also, $\psi^{2}\in G^{+}$. Thus $\Phi^{2}\cap G^{+}$
cannot be contained in a 1-parameter subgroup.

This completes the proof of Proposition \ref{prop:large-free-sets}.

\section{\label{sec:Extensions-and-open-problems}One more variation}

We conclude with a variation on Conjecture \ref{conj:main} in the
non-linear setting, where we as yet are unable to prove even the analog
of Theorem \ref{thm:main}:
\begin{problem}
\label{prob:non-linear}Let $\Phi\subseteq C^{\omega}([0,1])$ be
a compact set of contracting real-analytic maps of $[0,1]$. Let $X$
denote the attractor of $\Phi$. If $\dim\Phi>0$ and $X$ is not
a singleton, is $\dim X=1$?
\end{problem}
This question is not well posed because $C^{\omega}$ does not carry
a canonical metric through which to define the condition dim$\Phi>0$.
One can easily imagine suitable definitions, though, for example we
could ask the dimension to be positive in the $C^{1}$-metric, or
as a subset of $L^{2}$. Note that formulating the problem for $C^{\alpha}$,
$1\leq\alpha\leq\infty$, poses some difficulty, since any proper
compact subset $X\subseteq[0,1]$ admits a positive dimension set
of $C^{\alpha}$-maps preserving it - namely maps which are the identity
on $X$ and act only in its complement. Working with analytic maps
eliminates this problem and is in any case the most likely case to
be true. 

Certain aspects of our proof carry over to this setting, in particular,
one can linearize the action, and to some extent obtain an analog
of Proposition \ref{prop:iterated-entropy-formula-for-action}. One
technical difficulty here is the absence of a dyadic-like partition
of the ambient vector space; one can introduce refining partitions
which at each stage consist of cells comparable to a ball, but each
cell will split into countably many sub-cells at each stage. This
means that the iterated entropy formulas from Section \ref{sub:Multiscale-formulas-for-entropy}
become rather useless, because, while formally correct, all the entropy
of components could be concentrated at a negligible fraction of levels.
It may be possible to overcome this by looking for a partition of
$\Phi$ rather than the whole space, but this does not completely
solve the problem.

Another crucial issue is that, even when a suitable partition of $\Phi$
can be found, the analogue of Lemma \ref{lem:Lipschitz-constant-away-from-diagonal}
(and consequently Lemma \ref{lem:entropy-of-image}) may be false.
That lemma was based on the fact that the map $f\mapsto(f(x_{1}),\ldots,f(x_{k}))$
determines $f$ when $x_{i}$ are well-separated points and $k$ is
large enough. This occurs when $\Phi$ is contained in a finite-dimensional
parameter space, but in general it will fail. The only remedy we know
of at present is to make some finite-dimensionality assumptions which
are rarely satisfies. In fact the only nontrivial application of these
ideas at present appears in \cite{HochmanSolomyak2016}, where this
strategy was applied to stationary measures on the projective line
under the (projective) action of $SL_{2}(\mathbb{R})$. In general
Problem \ref{prob:non-linear} remains open.

\bibliographystyle{plain}

\bigskip
\footnotesize

\noindent Einstein Institute of Mathematics\\
Edmond J. Safra Campus (Givat Ram)\\
The Hebrew University of Jerusalem\\
Jerusalem, 9190401\\
Israel

\medskip
\noindent Email: \texttt{mhochman@math.huji.ac.il}

\end{document}